\documentclass[12pt]{amsart}

\usepackage{amssymb}
\usepackage{graphicx}
\usepackage{fullpage}

\usepackage{colortbl}
\providecommand{\mk}{\cellcolor[gray]{.75}}

\newtheorem{theorem}{Theorem}[section]
\theoremstyle{plain}

\newtheorem{corollary}[theorem]{Corollary}

\newtheorem{lemma}[theorem]{Lemma}

\newtheorem{proposition}[theorem]{Proposition}
\newtheorem{remark}[theorem]{Remark}

\numberwithin{equation}{section}

\def\spacer{\vline depth 0pt height 10pt width 0pt}

\def\ga{\circ}
\def\gb{\ast}
\def\gc{\cdot} 
\def\gd{\bullet}
\newcommand{\diff}{\mathop{\mathrm{diff}}}
\newcommand{\dist}{\mathop{\mathrm{dist}}}
\def\cst{\delta} 
\def\cnb{\Delta} 
\def\sdp{\rtimes}

\def\swr{\hbox{$\setminus\!\!\setminus$}}

\def\aut{\mathrm{Aut}}
\def\dom{\mathrm{Dom}}
\def\im{\mathrm{Im}}

\def\Dic{\mathrm{Dic}}

\def\lref#1{Lemma~$\ref{#1}$}
\def\pref#1{Proposition~$\ref{#1}$}
\def\tref#1{Theorem~$\ref{#1}$}
\def\fref#1{Figure~$\ref{#1}$}
\def\cref#1{Corollary~$\ref{#1}$}

\renewcommand{\ge}{\geqslant}
\renewcommand{\le}{\leqslant}

\def\good#1{\underline{#1}}
\def\best#1{\underline{\mathbf{#1}}}
\def\narrowcols{\addtolength{\arraycolsep}{-1.5pt}}
\def\normalcols{\addtolength{\arraycolsep}{1.5pt}}

\title{Closest multiplication tables of groups}

\author{Petr Vojt\v{e}chovsk\'y}
\address[Vojt\v{e}chovsk\'y]{Department of Mathematics, University of Denver, Colorado 80208, USA}
\email[Vojt\v{e}chovsk\'y]{petr@math.du.edu}

\author{Ian M. Wanless}
\address[Wanless]{School of Mathematical Sciences, Monash University, VIC 3800 Australia}
\email[Wanless]{ian.wanless@monash.edu}

\begin{document}

\begin{abstract}
Suppose that all groups of order $n$ are defined on the same set $G$ of cardinality $n$, and let the \emph{distance} of two groups of order $n$ be the number of pairs $(a,b)\in G\times G$ where the two group operations differ. Given a group $G(\ga)$ of order $n$, we find all groups of order $n$, up to isomorphism, that are closest to $G(\ga)$.
\end{abstract}



\thanks{This work was partially supported by a grant from the Simons Foundation (grant 210176 to Petr Vojt\v{e}chovsk\'y) and by the Australian Research Council (grants DP0662946 and DP1093320 to Ian Wanless).}

\maketitle

\section{Introduction}

Let $G$ be a finite set of cardinality $n$, and let $\ga$, $\gb$, $\gc$, $\gd$ be group operations defined on $G$. For groups $G(\ga)$, $G(\gb)$, let
\begin{align*}
    \diff(\ga,\gb) &= \{(a,b)\in G\times G;\;a\ga b\ne a\gb b\},\\
    \dist(\ga,\gb) &= |\diff(\ga,\gb)|,
\end{align*}
and call $\dist(\ga,\gb)$ the \emph{(Hamming) distance of groups $G(\ga)$, $G(\gb)$}.

In a research programme spanning two decades, Ale\v{s} Dr\'apal showed that there is a strong relationship between algebraic properties of groups and their distances, as will become apparent from many of his results we quote below.

In this paper we solve the following problem: \emph{Given a group $G(\ga)$, determine all multiplication tables of groups $G(\gb)$
(up to isomorphism) that are as close to the multiplication table of $G(\ga)$ as possible.}
More formally, let
\begin{align*}
    \cst(\ga) &= \min\{\dist(\ga,\gb);\;G(\ga)\ne G(\gb)\},\\
    \cnb(\ga) &= \{G(\gb);\;\dist(\ga,\gb) = \cst(\ga)\}.
\end{align*}
Our task is then to find $\cst(\ga)$ and to construct one group $G(\gb)$ of minimum distance from $G(\ga)$ for every isomorphism class of groups intersecting $\cnb(\ga)$.\

In particular, we determine the minimal distance
\begin{align*}
    \cst(n) = \min\{\cst(\ga);\;G(\ga)\text{ is a group of order }n\}
\end{align*}
and all pairs of groups $G(\ga)$, $G(\gb)$ (up to isomorphism) of order $n$ satisfying $\dist(\ga,\gb)=\cst(n)$.

\subsection{The context}

Let
\begin{align*}
    \cst_{\cong}(\ga) &= \min\{\dist(\ga,\gb);\;G(\ga)\cong G(\gb)\ne G(\ga)\},\\
    \cst_{\not\cong}(\ga) &= \min\{\dist(\ga,\gb);\;G(\ga)\not\cong G(\gb)\},
\end{align*}
where the second quantity is set to $\infty$ if all groups of order $n$ are isomorphic. Obviously, we have $\cst(\ga) = \min\{\cst_{\cong}(\ga),\,\cst_{\not\cong}(\ga)\}$.

An important threshold for $\cst(\ga)$ is obtained by considering pairs of groups isomorphic via a transposition. Note that if $f=(a,b)$ is an isomorphism between $G(\ga)$ and $G(\gb)$ then $\diff(\ga,\gb)$ is a subset of the rows and columns indexed by $a$, $b$, and of the ``diagonal'' entries $(x,y)$ with $x\ga y\in\{a,b\}$. This means that $\delta(n)$ will not exceed $6n$. More precisely:

As in \cite{DrEJC}, for a nontrivial commutative group $O$ of odd order, let $D(O)$ be the generalized dihedral group defined on $O\times C_2$ by
\begin{displaymath}
    (a,0)(b,h) = (ab,h),\quad (a,1)(b,h) = (ab^{-1},1+h).
\end{displaymath}
Then let
\begin{equation}\label{Eq:Delta0}
    \cst_0(\ga) = \left\{\begin{array}{ll}
        6n-18,&\text{if $n$ is odd},\\
        6n-20,&\text{if $G(\ga)\cong D(O)$ for some $O$},\\
        6n-24,&\text{otherwise}.
    \end{array}\right.
\end{equation}

The main results of \cite{DrEJC} can be summarized as follows:

\begin{theorem}[Dr\'apal]\label{Th:DrEJC} Let $|G|=n$ and let $G(\ga)$, $G(\gb)$ be groups defined on $G$. If $\dist(\ga,\gb)<n^2/9$ then $G(\ga)$ and $G(\gb)$ are isomorphic. If $n\ge 5$ then $\dist(\ga,\gb)\ge\cst_0(\ga)$ whenever $G(\gb)$ is isomorphic to $G(\ga)$ via a transposition, and $\dist(\ga,\gb) = \cst_0(\ga)$ for some $G(\gb)$ isomorphic to $G(\ga)$ via a transposition. Consequently, if $n\ge 51$ then $\cst(\ga) = \cst_0(\ga) = \cst_{\cong}(\ga) < \cst_{\not\cong}(\ga)$.
\end{theorem}

Moreover, \cite[Proposition 5.8]{DrEJC} describes in detail the transpositions that achieve the distance $\cst_0(\ga)$. Hence our problem has already been solved in all but finitely many cases. Here is an overview of other known results concerning distances of groups:

To determine $\cst_{\not\cong}(\ga)$ appears to be a very difficult problem. We already know from \tref{Th:DrEJC} that $\cst_{\not\cong}(\ga)\ge n^2/9$ whenever $n\ge 5$. When $G(\ga)$ is a $2$-group then $\cst_{\not\cong}(\ga)\ge n^2/4$ by \cite{Dr2}. Examples of non-isomorphic $2$-groups at \emph{quarter distance}, that is, with $\dist(\ga,\gb) = n^2/4$, can be found in \cite{DrConstr1} and \cite{DrConstr2}. In \cite{DrProc}, Dr\'apal constructed a family of $p$-groups for every prime $p>2$ with the property $\cst_{\not\cong}(\ga) = (n^2/4)(1-1/p^2)$. In particular, there is a $3$-group satisfying $\cst_{\not\cong}(\ga) = 2n^2/9$ (see also Construction 2 in Subsection \ref{Ss:Other}). Ivanyos {\em et al.} \cite{ILY} showed, after this paper had been submitted, that $\cst_{\not\cong}(\ga)\ge 2n^2/9$ always holds.


Let $\mathcal G(n)$ be a graph whose vertices are the isomorphism classes of groups of order $n$, and in which two vertices, possibly the same, form an edge if and only if they contain representatives at distance $\cst(n)$.

When $n$ is a power of two, let $\mathcal G'(n)$ be a graph on the same vertices as $\mathcal G(n)$ in which two vertices, possibly the same, form an edge if an only if they contain representatives at distance $n^2/4$ obtained by one of the two constructions of Dr\'apal \cite{DrConstr1} that we recall in Subsection \ref{Ss:Quarter}. When $n\in\{8,16\}$, it turns out that $\cst(n)=n^2/4$, so $\mathcal G'(n)$ is a subgraph of $\mathcal G(n)$.

By \cite{DrDiscr}, $\cst(\ga)\ge n^2/4$ for any $2$-group $G(\ga)$ of order $n\le 16$. In \cite{VoMS, VoProc}, the first author determined the connected graph $\mathcal G(8)$ with $\cst(n)=8^2/4=16$ (we checked that $\mathcal G'(8) = \mathcal G(8)$), calculated $\cst(\ga)$ for cyclic groups $G(\ga)$ of order less than $13$, proved that $\cst(\ga)=6n-18$ whenever $G(\ga)$ is a group of prime order $n>7$, and constructed a class of groups with $\cst(\ga)<\cst_0(\ga)$, of which the largest member has order $21$. (As we are going to show, $n=21$ happens to be the largest order for which $\cst(\ga)<\cst_0(\ga)$ can occur.)

B\'alek \cite{Ba} computed the subgraph $\mathcal G'(16)$ (excluding the diagonal entries) of $\mathcal G(16)$. Since $\mathcal G'(16)$ turns out to be connected, it follows that $\cst(\ga) = n^2/4$ for every group $G(\ga)$ of order $n=16$. A more direct argument establishing the connectedness of $\mathcal G(16)$ can be found in \cite{DrZh}. Our computational results show that $\mathcal G'(16) = \mathcal G(16)$. The two constructions of Subsection \ref{Ss:Quarter} can therefore be seen as canonical for $n\in\{8,16\}$.

Groups at quarter distance received attention even for orders $n=2^k>16$, although then $\cst(n)<n^2/4$ so $\mathcal G'(n)$ is no longer a subgraph of $\mathcal G(n)$. In \cite{Zh}, Zhukavets calculated $\mathcal G'(32)$ and $\mathcal G'(64)$; the first graph is connected while the second one has two connected components.

The quarter distance is of interest outside the variety of groups, too. In \cite{DrVo}, Dr\'apal and the first author generalized the constructions of \cite{DrConstr1} for \emph{Moufang loops}, that is, loops satisfying the identity $x(y(xz)) = ((xy)x)z$. The first author went on to construct a large family of Moufang loops of order $64$ \cite{VoEJC}, starting with the well-known Moufang loops $M_{2n}(G,2)$ of Chein \cite[pp. 35--38]{Chein} and using the constructions of \cite{DrVo}. Nagy and the first author eventually proved in \cite{NaVo} that the family of \cite{VoEJC} actually contains all Moufang loops of order $64$ up to isomorphism.

Distances of infinite groups are somewhat trivial, as it was shown in \cite{DrEJC} that if $G(\ga)$ is a group of infinite cardinality $\kappa$ then $\cst_{\cong}(\ga) = \cst_{\not\cong}(\ga) = \kappa$.

\subsection{The content}

For the convenience of the reader, the main result is stated at the outset in Section \ref{Sc:Main}.

For two subsets $\mathcal A$, $\mathcal B$ of groups defined on $G$, let
\begin{displaymath}
    \dist(\mathcal A,\mathcal B) = \min\{\dist(\ga,\gb);\;G(\ga)\in \mathcal A,\,G(\gb)\in\mathcal B,\,G(\ga)\ne G(\gb)\}.
\end{displaymath}
Denote by $[\ga]$ the class of all groups defined on $G$ and isomorphic to $G(\ga)$. In Section \ref{Sc:Aut}, we recall that $\dist([\ga],[\gb])= \dist([\ga],\gb)$. Consequently, the values of $\cst(\ga)$, $\cst_{\cong}(\ga)$ and $\cst_{\not\cong}(\ga)$ depend only on the isomorphism type of $G(\ga)$. If $n\ge 5$, \lref{Lm:Neutral} allows us to assume that closest groups have the same neutral element. \lref{Lm:Aut} shows how automorphism groups of $G(\ga)$, $G(\gb)$ come into play to speed up the calculation of $\dist([\ga],[\gb])$.

In Section \ref{Sc:Tools} we introduce, following Dr\'apal, these concepts and parameters:
\begin{align} \label{Eq:Tools}
\begin{split}
\textstyle \diff_a(\ga,\gb) &=
\{(a,b);\;b\in G,\,a\ga b\ne a\gb b\}|,\quad
\textstyle \dist_a(\ga,\gb) = |\diff_a(\ga,\gb)|,\\
    m(\ga,\gb) &=
\textstyle
\min\{\dist_a(\ga,\gb);\;a\in G,\,\dist_a(\ga,\gb)>0\},\\
    H(\ga,\gb) &=
\textstyle
\{a\in G;\;\dist_a(\ga,\gb)=0\},\quad h(\ga,\gb)=|H(\ga,\gb)|,\\
    K(\ga,\gb) &=
\textstyle
\{a\in G;\;\dist_a(\ga,\gb)<n/3\},\quad k(\ga,\gb)=|K(\ga,\gb)|.
\end{split}
\end{align}
When $\ga$, $\gb$ are fixed, we drop the operations from the names of the parameters and write $\dist_a$, $m$, $H$, and so on.

Among other results, we recall in Section \ref{Sc:Tools} that $a\ga b\ne a\gb b$ implies $\dist_a+\dist_b + \dist_{a\ga b}\ge n$; the set $H$ is either empty or it is a subgroup of both $G(\ga)$ and $G(\gb)$; if $|k|>3n/4$ then $\dist(\ga,\gb)>\cst_0(\ga)$; $m\ge 2$ if $n$ is even and $m\ge 3$ if $n$ is odd. We also study $\dist_a$ when the orders of $a$ in $G(\ga)$ and $G(\gb)$ disagree.

Building on these results, in Section \ref{Sc:Inequalities} we develop a series of inequalities relating $n$, $h$, $k$, $m$ and, consequently, we find only a few (less than hundred) quadruples $(n,h,k,m)$ in the range $22<n<51$ that can possibly yield $\dist(\ga,\gb)\le \cst_0(\ga)$. This will already imply that $\dist(\ga,\gb)<\cst_0(\ga)$ cannot hold for $n\ge 43$, improving upon the bound $n\ge 51$ of \tref{Th:DrEJC}.

In Section \ref{Sc:m2}, we first show that the case $m=2$ can be reduced to the study of distances of the cyclic group $C_n$ from a group possessing an element of order $n/2$, a case that is not difficult to handle computationally. We can proceed similarly when $n$ is a prime, independently verifying the results of \cite{VoMS, VoProc}.

The general algorithm for finding $\dist([\ga],[\gb])$ is given in Section \ref{Sc:Algorithm}. The algorithm is sufficiently fast to deal with all orders $n\le 22$ and also all cases when $h>1$, leaving us with only $20$ quadruples $(n,h,k,m)$, which require a very delicate analysis.

In Section \ref{Sc:Graph} we study the question: \emph{Given an edge-colored graph on $v$ vertices such that no color is used more than $m$ times and no vertex is adjacent to more than two edges of the same color, how many edges must the graph have to guarantee a rainbow $i$-matching?} A partial answer can be found in \pref{Pr:Graph}.

Returning to the problem of group distances, in Section \ref{Sc:ApplyGraph} we study the set $\{(a,b)\in\diff(\ga,\gb);\;a\in K,b\not\in K,a\ga b\not\in K\}$ and similar sets which give rise to edge-colored graphs. The main idea of Section \ref{Sc:ApplyGraph} is to exhibit a large enough rainbow matching in a certain graph to push the distance over the threshold $\cst_0(\ga)$.

Only $7$ quadruples $(n,h,k,m)$ remain after this analysis, all with $n\le 28$. These are disposed of in Section \ref{Sc:Stubborn}, using a series of increasingly more specialized lemmas.

Finally, in Section \ref{Sc:Constructions} we present several constructions that produce all pairs $G(\ga)$, $G(\gb)$ with $\dist(\ga,\gb) = \cst(\ga)<\cst_0(\ga)$. These are the constructions alluded to in \tref{Th:Main}, the main result.

\section{Main result}\label{Sc:Main}

\begin{theorem}\label{Th:Main}
Let $G$ be a set of size $n\ge 4$. Let $G(\ga)$ be a group defined on $G$, $\cst(\ga) = \min\{\dist(\ga,\gb);\;G(\gb)\text{ is a group different from }G(\ga)\}$, $\cnb(\ga) = \{G(\gb);\;\dist(\ga,\gb)=\cst(\ga)\}$, and let $\cst_0(\ga)$ be defined as in \eqref{Eq:Delta0}.

Then the value of $\cst(\ga)$ and one representative from $\cnb(\ga)$
for every isomorphism type of groups present in $\cnb(\ga)$ can be
found as follows:
\begin{enumerate}
\item[$\bullet$] If $n\not\in\{4,6,7,8,9,10,12,14,15,16,18,21\}$ then
  $\cst(\ga)=\cst_0(\ga)$, all groups in $\cnb(\ga)$ are isomorphic to
  $G(\ga)$, and there is a transposition $f$ of $G$ such that
  $f:G(\ga)\to G(\gb)$ is an isomorphism and $G(\gb)\in \cnb(\ga)$.
\item[$\bullet$] Otherwise the value of $\cst(\ga)$ and the
  isomorphism types of groups in $\cnb(\ga)$ can be found in Table
  \ref{Tb:Main}. When $n$ is a power of two and also in the case $\dist(C_3\times S_3,C_3\times S_3)$, the representatives of
  $\cnb(\ga)$ can be obtained by the constructions
  of Subsection $\ref{Ss:Quarter}$. When $n$ is not a power of two, the
  representatives of $\cnb(\ga)$ can be obtained by one of the three
  types of constructions of Subsection $\ref{Ss:Other}$, as indicated by
  the superscript in the table.
\end{enumerate}
In particular,
\begin{enumerate}
\item[$\bullet$] $\cst(\ga)<\cst_0(\ga)$ if and only if $G(\ga)$ is one of the following groups: $C_6$, $C_{10}$, $C_{14}$, $C_{21}$, a group of order $12$ except for $A_4$, a group of order $7$, $8$, $9$, $15$, $16$ or $18$.

\item[$\bullet$] $\cnb(\ga)$ contains groups of more than one isomorphism type if and only if $G(\ga)$ is one of the following groups: $C_9$, $D_{10}$, a group of order $8$, a group of order $16$, $D_{18}$, $C_{18}$, $C_6\times C_3$.

\item[$\bullet$] $\cnb(\ga)$ contains no groups isomorphic to $G(\ga)$ if and only if $G(\ga)$ is one of the following groups: $C_4$, $(C_2)^2$, $S_3$, $Q_8$, $(C_2)^3$, $(C_3)^2$, $(C_2)^4$, $(C_3)^2\sdp C_2$.

\end{enumerate}
\end{theorem}

\begin{scriptsize}
\begin{table}
\caption{Distances of isomorphism classes of groups for all orders $n$ where at least one group $G(\ga)$ satisfies $\cst(\ga)<\cst_0(\ga)$. A group of order $n$ labeled by $i$ is the $i$th group of order $n$ as listed in \texttt{GAP}. The row labels are structural descriptions of the groups with the usual conventions. The distance $\dist([\ga],[\gb])$ between the $i$th group $G(\ga)$ and the $j$th group $G(\gb)$ of order $n$ can be found in row $i$ and column $j$ of the table for $n$. This value is underlined if it is less than $\cst_0(\ga)$ (this has the potential to break the diagonal symmetry of the tables but actually never does), it is in bold face if it equals $\cst(n)$, and it is replaced with ``?'' if it was not calculated exactly but exceeds $\cst_0(\ga)$. The superscript points to a construction in Subsection \ref{Ss:Other} that achieves the distance.}\label{Tb:Main}

\begin{displaymath}
\narrowcols
\begin{array}{r|c|c}
n=4
&1&2\\
\hline\spacer
C_4=1&7&\bf{4}\\
(C_2)^2=2&\bf{4}&16\\
\end{array}
\normalcols
\quad
\narrowcols
\begin{array}{r|c|c}
n=6
&1&2\\
\hline\spacer
S_3=1&16&12^{1}\\
C_6=2&12&\best{8}^2\\
\end{array}
\normalcols
\quad
\narrowcols
\begin{array}{r|c|c}
n=9
&1&2\\
\hline\spacer
C_9=1&\best{18}^3&\best{18}^2\\
(C_3)^2=2&\best{18}^2&36\\
\end{array}
\normalcols
\end{displaymath}

\begin{displaymath}
\narrowcols
\begin{array}{r|c}
n=7
&1\\
\hline\spacer
C_7=1&\best{18}^3\\
\end{array}
\normalcols
\quad\quad
\narrowcols
\begin{array}{r|c}
n=15
&1\\
\hline\spacer
C_{15}=1&\best{50}^2\\
\end{array}
\end{displaymath}

\begin{displaymath}
\narrowcols
\begin{array}{r|c|c}
n=10
&1&2\\
\hline\spacer
D_{10}=1&40&40^1\\
C_{10}=2&40&\best{24}^2\\
\end{array}
\normalcols
\quad
\narrowcols
\begin{array}{r|c|c}
n=14
&1&2\\
\hline\spacer
D_{14}=1&64&84\\
C_{14}=2&84&\best{48}^2\\
\end{array}
\normalcols
\quad
\narrowcols
\begin{array}{r|c|c}
n=21
&1&2\\
\hline\spacer
C_7\sdp C_3=1&108&?\\
C_{21}=2&?&\best{98}^2\\
\end{array}
\normalcols
\end{displaymath}

\begin{displaymath}
\narrowcols
\begin{array}{r|c|c|c|c|c}
n=8
&1&2&3&4&5\\
\hline\spacer
C_8=1           &\best{16}&\best{16}&24&24&28\\
C_4\times C_2=2 &\best{16}&\best{16}&\best{16}&\best{16}&\best{16}\\
D_8=3           &24&\best{16}&\best{16}&\best{16}&\best{16}\\
Q_8=4           &24&\best{16}&\best{16}&24&24\\
(C_2)^3=5       &28&\best{16}&\best{16}&24&24\\
\end{array}
\normalcols
\end{displaymath}

\begin{displaymath}
\narrowcols
\begin{array}{r|c|c|c|c|c}
n=12
&1&2&3&4&5\\
\hline\spacer
\Dic_3=1         &\best{32}^2&48&82&\good{36}&60\\
C_{12}=2        &48&\best{32}^2&70&60&\good{36}\\
A_4=3           &82&70&48&72&60\\
D_{12}=4        &\good{36}&60&72&\best{32}^2&48\\
C_6\times C_2=5 &60&\good{36}&60&48&\best{32}^2\\
\end{array}
\normalcols
\quad\quad
\narrowcols
\begin{array}{r|c|c|c|c|c}
n=18
&1&2&3&4&5\\
\hline\spacer
D_{18}=1                & \best{72}^3&144&144&\best{72}^2&180\\
C_{18}=2                &144& \best{72}^3&138&180&\best{72}^2\\
C_3\times S_3=3         &144&138& \good{81}&108&108\\
(C_3)^2\sdp C_2=4           & \best{72}^2&180&108& 88&144\\
C_6\times C_3=5         &180& \best{72}^2&108&144& \best{72}^2\\
\end{array}
\normalcols
\end{displaymath}

\begin{displaymath}
\narrowcols
\begin{array}{r|c|c|c|c|c|c|c|c|c|c|c|c|c|c}
n=16
&1&2&3&4&5&6&7&8&9&10&11&12&13&14\\
\hline\spacer
C_{16}=1 & \best{64}& \best{64}&112&112& \best{64}& 96&112&112&112&112&136&136&128&148\\
(C_4)^2=2 & \best{64}& \best{64}& \best{64}& \best{64}& \best{64}& 88&128&112&112& \best{64}& 96& 96& 96&112\\
\text{rank\ 2\ }(C_4\times C_2)\sdp C_2=3 &112& \best{64}& \best{64}& \best{64}& 88& \best{64}& 96& \best{64}& 96& \best{64}& \best{64}& 96& 96& 96\\
C_4\sdp C_4=4&112& \best{64}& \best{64}& \best{64}& 88& \best{64}& 96& 96& \best{64}& \best{64}& \best{64}& \best{64}& 96&112\\
C_8\times C_2=5 & \best{64}& \best{64}& 88& 88& \best{64}& \best{64}& 96& 96& 96& \best{64}& 96& 96& 96&112\\
C_8\sdp C_2=6 & 96& 88& \best{64}& \best{64}& \best{64}& \best{64}& 96& 96& 96& 88& 96& 96& \best{64}&128\\
D_{16}=7 &112&128& 96& 96& 96& 96& \best{64}& \best{64}& \best{64}&112& \best{64}&112& 96&112\\
QD_{16}=8 &112&112& \best{64}& 96& 96& 96& \best{64}& \best{64}& \best{64}&112& 96& 96& \best{64}&128\\
Q_{16}=9 &112&112& 96& \best{64}& 96& 96& \best{64}& \best{64}& \best{64}&112& 96& \best{64}& 96&136\\
C_4\times (C_2)^2=10&112& \best{64}& \best{64}& \best{64}& \best{64}& 88&112&112&112& \best{64}& \best{64}& \best{64}& \best{64}& \best{64}\\
C_2\times D_8=11&136& 96& \best{64}& \best{64}& 96& 96& \best{64}& 96& 96& \best{64}& \best{64}& \best{64}& \best{64}& \best{64}\\
C_2\times Q_8=12&136& 96& 96& \best{64}& 96& 96&112& 96& \best{64}& \best{64}& \best{64}& \best{64}& \best{64}& 96\\
\text{rank\ 3\ }(C_4\times C_2)\sdp C_2 = 13&128& 96& 96& 96& 96& \best{64}& 96& \best{64}& 96& \best{64}& \best{64}& \best{64}& \best{64}& 88\\
(C_2)^4=14&148&112& 96&112&112&128&112&128&136& \best{64}& \best{64}& 96& 88& 72
\end{array}
\normalcols
\end{displaymath}
\end{table}
\end{scriptsize}

\subsection{Additional results}

The values $\cst_{\cong}(C_n)$ for $4\le n\le 22$ are as follows:
\begin{displaymath}
\begin{array}{|c|ccccccccccccccccccc|}
\hline
n&4&5&6&7&8&9&10&11&12&13&14&15&16&17&18&19&20&21&22\\
\cst_{\cong}(C_n)&7&12&8&18&16&18&24&48&32&60&48&50&64&84&72&96&96&98&108\\
\hline
\end{array}
\end{displaymath}

The distances for $n\in\{20,22\}$ are as follows, with the same notational conventions as in Table \ref{Tb:Main}:
\begin{displaymath}
\begin{array}{r|c|c|c|c|c}
n=20
&1&2&3&4&5\\
\hline\spacer
\Dic_5=1                 & \mathbf{96}&  ?& ?&100&?\\
C_{20}=2                &  ?& \mathbf{96}& ?&  ?&100\\
C_5\sdp C_4=3               &  ?&  ?&\mathbf{96}&  ?&?\\
D_{20}=4                &100&  ?& ?&\mathbf{96} &160\\
C_{10}\times C_2=5      &  ?&100& ?&  160&\mathbf{96}\\
\end{array}
\normalcols
\quad\quad
\narrowcols
\begin{array}{r|c|c}
n=22
&1&2\\
\hline\spacer
D_{22}&112&?\\
C_{22}&?&\mathbf{108}\\  
\end{array}
\quad\quad
\end{displaymath}

\section{Distances of isomorphism classes}\label{Sc:Aut}

For a group $G(\ga)$ and a bijection $f:G\to G$ there is a unique group $G(\gb)$ such that $f:G(\ga)\to G(\gb)$ is an isomorphism, namely $a\gb b = f(f^{-1}(a)\ga f^{-1}(b))$. We denote this operation $\gb$ by $\ga_f$.

\begin{lemma}\label{Lm:IsoMove}
Let $G(\ga)$, $G(\gb)$ be groups and $f:G\to G$ a bijection. Then $\dist_a(\ga,\gb) = \dist_{f(a)}(\ga_f,\gb_f)$ for every $a\in G$. In particular, $\dist(\ga,\gb) = \dist(\ga_f,\gb_f)$.
\end{lemma}
\begin{proof}
Fix $a\in G$. The cardinalities of the sets of elements $b\in G$ satisfying any of the following conditions are the same:
\begin{align*}
    a\ga b&\ne a\gb b,\\
    f^{-1}(f(a))\ga b&\ne f^{-1}(f(a))\gb b,\\
    f^{-1}(f(a))\ga f^{-1}(b)&\ne f^{-1}(f(a))\gb f^{-1}(b),\\
    f(f^{-1}(f(a))\ga f^{-1}(b))&\ne f(f^{-1}(f(a))\gb f^{-1}(b)),\\
    f(a)\ga_f b&\ne f(a)\gb_f b.
\end{align*}
\end{proof}

\begin{proposition}\label{Pr:IsoDist}
Let $G(\ga)$, $G(\gb)$ be groups. Then $\dist([\ga],[\gb]) = \dist([\ga],\gb)$. Moreover, if $G(\ga)\cong G(\gb)$ then $\cst(\ga) = \cst(\gb)$, $\cst_{\cong}(\ga) = \cst_{\cong}(\gb)$ and $\cst_{\not\cong}(\ga) = \cst_{\not\cong}(\gb)$.
\end{proposition}
\begin{proof}
Let $f$, $g:G\to G$ be bijections for which $\dist([\ga],[\gb]) = \dist(\ga_f,\gb_g)$. Then, by \lref{Lm:IsoMove}, $\dist([\ga],[\gb]) = \dist(\ga_f,\gb_g) = \dist((\ga_f)_{g^{-1}},\gb)\ge\dist([\ga],\gb)$. The other inequality is obvious.

Now assume that $\gb = \ga_f$ for some bijection $f:G\to G$, and let $G(\gc)$ be such that $\cst(\ga) = \dist(\ga,\gc)$. Then $\cst(\gb)\le \dist(\gb,\gc_f) = \dist(\ga_f,\gc_f) = \dist(\ga,\gc)=\cst(\ga)$, the other inequality follows by symmetry, so $\cst(\ga)=\cst(\gb)$. The equalities $\cst_{\cong}(\ga) = \cst_{\cong}(\gb)$ and $\cst_{\not\cong}(\ga) = \cst_{\not\cong}(\gb)$ are proved similarly.
\end{proof}

To determine $\dist([\ga],[\gb])$ it therefore suffices to find the minimal value of $\dist(\ga_f,\gb)$, where $f:G\to G$ is a bijection.

Let us denote the neutral element of $G(\ga)$ by $1(\ga)$, and the inverse of $a$ in $G(\ga)$ by $a^\ga$.

\begin{lemma}\label{Lm:Neutral}
Assume that $G(\ga)$, $G(\gb)$ have the same neutral element $1(\ga)=1(\gb)$, and let $f:G\to G$ be a bijection such that $\dist([\ga],[\gb]) = \dist(\ga_f,\gb)$. Then either $f(1(\ga)) = 1(\ga)$, or else $\ga_f =\gb_\ell$ for some transposition $\ell$ and $\dist([\ga],[\gb]) = \dist(\gb_\ell,\gb)$.
\end{lemma}
\begin{proof}
Let $G(\gc) = G(\ga_f)$, so $\dist([\ga],[\gb]) = \dist(\gc,\gb)$. Since $f:G(\ga)\to G(\gc)$ is an isomorphism, we have $1(\gc) = f(1(\ga))$. If $1(\gc)=1(\ga)$ we are done, so assume that $1(\gc) = f(1(\ga))\ne 1(\ga)$. Let $g = \ell\circ f$ be the composition of $f$ with the
transposition $\ell$ of $1(\ga)$ and $1(\gc)$, and let $G(\gd) = G(\ga_g)$. We claim that $\dist(\gd,\gb) < \dist(\gc,\gb)$.

Recall that $1(\ga)=1(\gb)$, and consider the set $E=\{(a,b)\in G\times G;\;\{a,b\}\cap\{1(\gc),1(\gb)\}\ne\emptyset\}$. We first show that $G(\gc)$ and $G(\gb)$ disagree on every entry of $E$. Indeed, if $a=1(\gc)$ and $b\in G$ then $a\gc b = 1(\gc)\gc b = b = 1(\gb)\gb b\ne 1(\gc)\gb b = a\gb b$, if $a=1(\gb)$ then $a\gc b = 1(\gb)\gc b \ne 1(\gc)\gc b = b = 1(\gb)\gb b = a\gb b$, and similarly if $b\in\{1(\gc),1(\gb)\}$. On the other hand, we claim that $G(\gd)$ and $G(\gb)$ agree on the row of $E$ indexed by $1(\gb)$, and on the column of $E$ indexed by $1(\gb)$. Indeed, we have $g^{-1}(1(\gb)) = f^{-1}(1(\gc)) = 1(\ga)$, and hence $1(\gb)\gd b = g(g^{-1}(1(\gb))\ga g^{-1}(b)) = g(1(\ga)\ga g^{-1}(b)) = g(g^{-1}(b)) = b = 1(\gb)\gb b$, and, similarly, $b\gd 1(\gb) = b\gb 1(\gb)$. Hence $|E\cap\diff(\gc,\gb)| - |E\cap\diff(\gd,\gb)|\ge 2n-1$.

Since the operation $\gd = \ga_g$ is obtained from $\gc = \ga_f$ by applying the transposition $\ell$, the two operations agree outside of $E$, except possibly on the two  ``diagonals''
\begin{displaymath}
    F=\{(a,b)\in G\times G;\;a\gc b = 1(\gb)\text{ or }a\gc b = 1(\gc)\}.
\end{displaymath}
Recall that $1(\gb)\gd b = 1(\gb)\gb b$ for every $b\in G$, in particular for the two values of $b$ with $(1(\gb),b)\in F$. Thus, in the worst case, $|F\cap\diff(\gc,\gb)| - |F\cap\diff(\gd,\gb)| \ge 0 - (|F|-2) = 2-2n$. We conclude that $\dist(\gd,\gb)<\dist(\gc,\gb)$.

This means that $\dist(\gd,\gb)=0$ and thus $\gd = \gb$. Since $\gd = \ga_g = (\ga_f)_\ell$, we see that $\ga_f = \gb_\ell$.
\end{proof}

While calculating $\dist([\ga],[\gb])$, we can certainly assume that $1(\ga) = 1(\gb) = 1$. Lemma \ref{Lm:Neutral} therefore allows us to consider only mappings $f$ fixing the element $1$, or to conclude that $\dist([\ga],[\gb]) = \dist(\gb_\ell,\gb)$ for some transposition $\ell$, a case fully resolved by Theorem \ref{Th:DrEJC} as long as $n\ge 5$. This speeds up the search slightly. A much larger improvement is achieved by looking at the automorphism groups of $G(\ga)$ and $G(\gb)$. Denote by $\aut(\ga)$ the automorphism group of $G(\ga)$.

\begin{lemma}\label{Lm:Aut}
Let $G(\ga)$, $G(\gb)$ be groups, $f:G\to G$ a bijection, and $g\in\aut(\ga)$, $\ell\in\aut(\gb)$. Then $\dist(\ga_f,\gb) = \dist(\ga_{\ell fg},\gb)$.
\end{lemma}
\begin{proof}
Note that $\ga_g = \ga$ and $\gb_\ell=\gb$. Using these facts and \lref{Lm:IsoMove}, we have $\dist(\ga_f,\gb) = \dist((\ga_g)_f,\gb) = \dist(\ga_{fg},\gb) = \dist((\ga_{fg})_\ell,\gb_\ell) = \dist(\ga_{\ell fg},\gb_\ell)$.
\end{proof}

\section{Structural tools}\label{Sc:Tools}

Recall the parameters \eqref{Eq:Tools}. The results \ref{Lm:Triple}--\ref{Lm:ConstantOnCosets} and \ref{Lm:HalfH}--\ref{Cr:BigK} are taken from \cite{DrEJC} and \cite{DrSzeged}, or are immediate corollaries of results therein. We do not hesitate to include short proofs here, and we refer the reader to \cite{DrEJC} and \cite{DrSzeged} for the longer, omitted proofs.

\begin{lemma}\label{Lm:Triple}
If $a\ga b\ne a\gb b$ then $\dist_a+\dist_b+\dist_{a\ga b}\ge n$.
\end{lemma}
\begin{proof}
Let $c\in G$ and suppose that $b\ga c = b\gb c$ and $(a\ga b)\ga c = (a\ga b)\gb c$. Then $a\ga (b\ga c) = (a\ga b)\ga c = (a\ga b)\gb c\ne (a\gb b)\gb c = a\gb(b\gb c) = a\gb(b\ga c)$.
\end{proof}

\begin{lemma}\label{Lm:HSubgroup}
Let $H=H(\ga,\gb)$. Then either $H=\emptyset$ or else $H\le G(\ga)$
and $H\le G(\gb)$.
\end{lemma}
\begin{proof}
Assume that $a$, $b\in H$. Then for every $c\in G$ we have $(a\ga b)\ga c = a\ga (b\ga c) = a\ga (b\gb c) = a\gb (b\gb c) = (a\gb b)\gb c = (a\ga b)\gb c$, so $a\ga b\in H$.
\end{proof}

We remark that, as per the previous section, we can always assume that $1\in
H$, so the case when $H=\emptyset$ will not arise in our work.

\begin{lemma}\label{Lm:ConstantOnCosets}
Suppose that $H\ne\emptyset$. If $b\in H\ga a$ then $\dist_a = \dist_b$.
\end{lemma}
\begin{proof}
Let $b=c\ga a$ for $c\in H$. Let $d\in G$ and suppose that $b\ga d=b\gb d$. Then $c\gb(a\ga d) = c\ga (a\ga d) = (c\ga a)\ga d = (c\ga a)\gb d = (c\gb a)\gb d = c\gb(a\gb d)$, and thus $a\ga d = a\gb d$. This shows that $\dist_a\le \dist_b$, and the other inequality follows from $a\in H\ga b$.
\end{proof}

\begin{lemma}\label{Lm:TwoCosets}
If $a\ga b\ne a\gb b$ then $H\ga b\ne H\ga(a\ga b)$.
\end{lemma}
\begin{proof}
If $H\ga b = H\ga(a\ga b)$ then $a = a\ga b\ga b^\ga\in H$, contradicting $\dist_a>0$.
\end{proof}

\begin{lemma}\label{Lm:HalfH}
If $h(\ga,\gb)=n/2$ then $\dist(\ga,\gb)\ge n^2/4$.
\end{lemma}

\begin{lemma}\label{Lm:HDividesK}
If $h>0$ then $h$ divides $k$.
\end{lemma}
\begin{proof}
Since the function $\dist:G\to \mathbb N$, $a\mapsto \dist_a$ takes on different values in $K$ and $G\setminus K$, \lref{Lm:ConstantOnCosets} implies that $K$ is a union of (right) cosets of $H$.
\end{proof}

\begin{proposition}\label{Pr:3n4}
If $k(\ga,\gb)>3n/4$ then there is an isomorphism $f:G(\ga)\to G(\gb)$
fixing all elements of $K(\ga,\gb)$.
\end{proposition}

The following example shows that \pref{Pr:3n4} is best possible. Let
$\ga,\gb$ be defined as follows, where differences are shaded.
\[
\begin{array}{c|cccccccc}
\ga&1&2&3&4&5&6&7&8\\
\hline
1&1&2&3&4&5&6&7&8\\
2&2&1&4&3&6&5&8&7\\
3&3&4&\mk1&\mk2&\mk8&\mk7&6&5\\
4&4&3&\mk2&\mk1&\mk7&\mk8&5&6\\
5&5&6&\mk8&\mk7&3&4&2&1\\
6&6&5&\mk7&\mk8&4&3&1&2\\
7&7&8&6&5&2&1&\mk3&\mk4\\
8&8&7&5&6&1&2&\mk4&\mk3\\
\end{array}
\qquad
\begin{array}{c|cccccccc}
\gb&1&2&3&4&5&6&7&8\\
\hline
1&1&2&3&4&5&6&7&8\\
2&2&1&4&3&6&5&8&7\\
3&3&4&\mk2&\mk1&\mk7&\mk8&6&5\\
4&4&3&\mk1&\mk2&\mk8&\mk7&5&6\\
5&5&6&\mk7&\mk8&3&4&2&1\\
6&6&5&\mk8&\mk7&4&3&1&2\\
7&7&8&6&5&2&1&\mk4&\mk3\\
8&8&7&5&6&1&2&\mk3&\mk4\\
\end{array}
\]
In this example, $k=6=3n/4$, but the groups are not isomorphic;
$G(\ga)\cong C_4\times C_2$ and $G(\gb)\cong C_8$.

\begin{proposition}\label{Pr:ManyFixedPoints}
Assume that $n\ge 12$, and let $f:G(\ga)\to G(\gb)$ be a non-identity isomorphism with more than $2n/3$ fixed points. Then $\dist(\ga,\gb)\ge \cst_0(\ga)$.
\end{proposition}

\begin{corollary}\label{Cr:BigK}
Assume that $n\ge 12$. If $k(\ga,\gb)>3n/4$ then $G(\ga)\cong G(\gb)$ and $\dist(\ga,\gb)\ge\cst_0(\ga)$.
\end{corollary}

In our search for closest groups $G(\gb)$ to $G(\ga)$, we can therefore assume that $k\le3n/4$ when $n\ge 12$.

Denote by $L_a(\ga)$ the left translation by $a$ in $G(\ga)$, that is, $L_a(\ga)(b) = a\ga b$. Let $\beta_a(\ga,\gb) =  (L_a(\ga))^{-1}L_a(\gb)$. Then $\beta_a(\ga,\gb)(b) = b$ if and only if $a\ga b=a\gb b$, and thus $\dist_a(\ga,\gb)$ is the number of points  moved by $\beta_a(\ga,\gb)$.

\begin{lemma}\label{Lm:MinimalM}
Assume that $\dist_a = \dist_a(\ga,\gb)>0$. Then $\dist_a \ge 2$. If $\beta_a(\ga,\gb)$ is an even permutation then $\dist_a\ge 3$. In particular, if $n$ is odd then $\dist_a\ge 3$.
\end{lemma}
\begin{proof}
The case $\dist_a=1$ is impossible since $\beta_a$ cannot move precisely $1$ point. When $\beta_a$ is even, it is not a transposition, and hence it moves at least $3$ points. When $n$ is odd, the left translations $L_a(\ga)$, $L_a(\gb)$ are products of cycles of odd length, hence $\beta_a$ is an even permutation.
\end{proof}

Finally, we investigate $\dist_a(\ga,\gb)$ depending on whether $a$ has the same order in $G(\ga)$ and $G(\gb)$. Denote by $|a|_\ga$ the order of $a$ in $G(\ga)$. If $|a|_\ga = |a|_\gb$, we say that $a$ is \emph{order matched}, otherwise it is \emph{order mismatched}.

\begin{lemma}\label{Lm:OrderMismatched}
Assume that $\sigma=|a|_\ga > |a|_\gb = \tau$. Then $\dist_a(\ga,\gb)\ge (n/\sigma)\lceil \sigma/\tau\rceil \ge n/\tau$.
\end{lemma}
\begin{proof}
The left translation $L_a(\ga)$ is a product of $n/\sigma$ disjoint cycles of length $\sigma$, and $L_a(\gb)$ is a product of $n/\tau$ disjoint cycles of length $\tau<\sigma$. Consider a cycle $(b_0,\dots,b_{\sigma-1})$ of $L_a(\ga)$. By definition then, $a\ga b_i = b_{i+1\mod \sigma}$. Let us focus on $b_0$. Without loss of generality, there is a cycle $(c_0,\dots,c_{\tau-1})$ of $L_a(\gb)$ such that $b_0=c_0$. Let $i$ be the least integer with $1\le i\le \tau$ such that $b_i\ne c_{i\mod \tau}$. (Such an $i$ exists, since $c_{\tau\mod \tau} = c_0=b_0\ne b_\tau$.) Then $a\ga c_{i-1} = a\ga b_{i-1} = b_i\ne c_{i\mod \tau} = a\gb c_{i-1} = a\gb b_{i-1}$.

Hence, corresponding to the segment $b_0$, $\dots$, $b_\tau$, we found a
difference $a\ga b_j\ne a\gb b_j$ with $0\le j\le \tau-1$. Repeating this
argument shows that there must be $\lceil \sigma/\tau\rceil$ differences within
each of the $n/\sigma$ cycles of $L_a(\ga)$.
\end{proof}

By \tref{Th:DrEJC}, $\cst_{\cong}(\ga)<\cst_{\not\cong}(\ga)$ when $n\ge 51$. We can reach the same conclusion for some smaller orders $n$, too:

\begin{lemma}\label{Lm:2p}
Let $n=2p$ for a prime $p\ge 11$. Let $G(\ga)$ be a group of order $n$. Then $\cst_{\cong}(\ga)<\cst_{\not\cong}(\ga)$.
\end{lemma}
\begin{proof}
Up to isomorphism, there are only two groups of order $2p$, the cyclic group $C_{2p}=G(\ga)$ and the dihedral group $D_{2p}=G(\gb)$. There is a unique involution in $C_{2p}$ and there are $p$ involutions in $D_{2p}$. Hence at least $p-1$ involutions are order mismatched. By \lref{Lm:OrderMismatched}, $d_a(\ga,\gb)\ge 2p/2=p$ for every order mismatched involution $a$. We therefore have $\dist(\ga,\gb)\ge (p-1)p$. On the other hand, $\cst(C_{2p})\le 12p-24$ and $\cst(D_{2p})\le 12p-20$ by \tref{Th:DrEJC}. The inequality $(p-1)p>12p-20$ holds for every $p\ge 13$.

It remains to discuss the case $p=11$. If at least one element $a$ in the cyclic subgroup $C_p$ of $D_{2p}$ satisfies $\dist_a>0$ (hence $\dist_a\ge 2$), then the same inequality holds for every nonidentity element of $C_p$, by \lref{Lm:HSubgroup}, and thus $\dist(\ga,\gb)\ge (p-1)p+2(p-1)>12p-20$. Otherwise, $C_p=H$, and $\dist(\ga,\gb)\ge 2p^2>12p-20$ by \lref{Lm:HalfH}.
\end{proof}

\begin{lemma}\label{Lm:OrderMatched} If $\dist_a>0$ and $a$ is order matched then $\dist_a\ge 3$.
\end{lemma}
\begin{proof}
The two left translations $L_a(\ga)$ and $L_a(\gb)$ have the same cycle structure, thus $\beta_a(\ga,\gb)$ is an even permutation, and we are done by \lref{Lm:MinimalM}.
\end{proof}

We can now narrow down possible isomorphism types of $G(\ga)$ and
$G(\gb)$ when $m=2$.

\begin{proposition}\label{Pr:m2}
Assume that $\dist_a(\ga,\gb)=2$. Then, without loss of generality, $|a|_\ga = n$ and $|a|_\gb = n/2$.
\end{proposition}
\begin{proof}
Since $\dist_a=2$, $a$ must be order mismatched, by
\lref{Lm:OrderMatched}. Let $\sigma = |a|_\ga$ and
$\tau=|a|_\gb$. Without loss of generality, $\sigma>\tau$. Then, by
\lref{Lm:OrderMismatched},
$2=\dist_a\ge(n/\sigma)\lceil\sigma/\tau\rceil$. As $\sigma>\tau$, we
must have $n/\sigma=1$ and $\lceil \sigma/\tau\rceil = 2$, hence
$n=\sigma$, $\lceil n/\tau\rceil =2$, and because $\tau$ divides $n$,
it follows that $\tau=n/2$.
\end{proof}

For a group $G(\ga)$ and integer $\ell\ge 1$, let $o_\ell(\ga)$ be the number of elements of order $\ell$ in $G(\ga)$. Motivated by \pref{Pr:m2}, we let
\begin{displaymath}
    \omega(\ga,\gb) = \min\{o_n(\ga),o_{n/2}(\gb)\} + \min\{o_{n/2}(\ga),o_n(\gb)\}.
\end{displaymath}

Let $\varphi$ denote Euler's totient function.

\begin{lemma}\label{Lm:LimitOnm2}
For groups $G(\ga)$, $G(\gb)$ of even order $n$, there are at most $h(\ga,\gb)+2\varphi(n/2)$ rows $a\in G$ with $\dist_a<3$.
\end{lemma}
\begin{proof}
Consider $a\not\in H$. If $a$ is order matched, then $\dist_a\ge 3$ by
\lref{Lm:OrderMatched}. If $a$ is order mismatched and $\dist_a=2$, we
must have $\{|a|_\ga,|a|_\gb\} = \{n,n/2\}$, by
\lref{Lm:OrderMismatched}. The number of elements $a$ with
$\{|a|_\ga,|a|_\gb\}=\{n,n/2\}$ cannot exceed $\omega(\ga,\gb)$.  Thus
it suffices to show that $\omega(\ga,\gb)\le2\varphi(n/2)$.

Suppose $G(\ga)$ is not cyclic. Then
$\omega(\ga,\gb)=\min\{0,o_{n/2}(\gb)\}+\min\{o_{n/2}(\ga),o_n(\gb)\}
\le o_n(\gb)\le\varphi(n)\le2\varphi(n/2)$. A similar argument works if
$G(\gb)$ is not cyclic, so we may as well assume that both
$G(\ga)$ and $G(\gb)$ are cyclic. In that case
$\omega(\ga,\gb)=2\min\{o_{n/2}(\ga),o_n(\gb)\}
=2\min\{\varphi(n/2),\varphi(n)\}
=2\varphi(n/2)$.
\end{proof}

\section{Inequalities}\label{Sc:Inequalities}

We now start the search for closest multiplication tables of groups.

Let $G(\ga)$, $G(\gb)$ be two groups of order $n$, and let $h=h(\ga,\gb)$, $k=k(\ga,\gb)$, $m=m(\ga,\gb)$. Keeping our goal in mind, we can make the following assumptions on $n$, $h$, $k$ and $m$:
\begin{enumerate}
\item[-] $23\le n\le 50$ (the case $n\ge 51$ is covered by
  \tref{Th:DrEJC}, the case $n\le 22$ will be addressed later),
\item[-] $1\le h<n$ and $h$ divides $n$ (we can assume $1\le h$ by
  \lref{Lm:Neutral}, $h<n$ to avoid $G(\ga)=G(\gb)$, and $h$ divides
  $n$ by \lref{Lm:HSubgroup}),
\item[-] $k\le 3n/4$ and $h$ divides $k$ (by
  \cref{Cr:BigK} and \lref{Lm:HDividesK}),
\item[-] $m\ge 2$ when $n$ is even and $m\ge 3$ when $n$ is odd (by
  \lref{Lm:MinimalM}). By the definition of $k$, we also know $m<n/3$
  if $h<k$, whereas $n/3\le m\le n$ if $h=k$.
\end{enumerate}
We will consider quadruples $(n,h,k,m)$ satisfying the above
conditions. We are interested only in such quadruples for which
$\dist(\ga,\gb)\le\cst(\ga)$ occurs. Since we do not want to assume
(yet) anything about the isomorphism type of $G(\ga)$, we set
\begin{displaymath}
    \cst_0(n) = \left\{\begin{array}{ll}
        6n-18,&\text{ when $n$ is odd},\\
        6n-20,&\text{ when $n\equiv2\mod4$},\\
        6n-24,&\text{ when $n\equiv0\mod4$},
    \end{array}\right.
\end{displaymath}
and we keep only those quadruples for which it is possible that
$\dist(\ga,\gb)\le \cst_0(n)$.  We will eliminate most quadruples by a
series of inequalities.


We start with a fundamental inequality based on both $H$ and
$K$. Every element of $G\setminus K$ satisfies
$\dist_a\ge\lceil n/3\rceil$, and $H\subseteq K$, thus
\begin{equation}\label{Eq:I4}
    \dist(\ga,\gb)\ge (n-k)\lceil n/3\rceil + (k-h)m.
\end{equation}
There are $309$ quadruples $[n,h,k,m]$ that satisfy this constraint.
We will gradually whittle these away until none remain (at the end of Section \ref{Sc:Stubborn}).

Let $a$ be such that $\dist_a=m$. By \lref{Lm:Triple}, there is
$b$ such that $\dist_a+\dist_b+\dist_{a\ga b}\ge n$. Hence
$\dist_b+\dist_{a\ga b}\ge n-m$, and we conclude that there exists $c$
such that $\dist_c\ge \lceil (n-m)/2\rceil$. Then by
\lref{Lm:ConstantOnCosets}, there are (at least) $h$ elements $c$ with
$\dist_c\ge \lceil (n-m)/2\rceil$, all in $G\setminus H$. The
remaining $n-2h\ge 0$ elements of $G\setminus H$ satisfy $\dist_a\ge m$,
and we have
\begin{equation}\label{Eq:I2}
    \dist(\ga,\gb) \ge h\Big\lceil\frac{n-m}{2}\Big\rceil + (n-2h)m.
\end{equation}
(282 quadruples remain.)

By \lref{Lm:HalfH},
\begin{equation}\label{Eq:I3}
    \text{if $h=n/2$ then }\dist(\ga,\gb)\ge n^2/4.
\end{equation}
(207 quadruples remain, all with $m < n/3$ and $h<k$.)

Let again $a\ga b\ne a\gb b$, and assume $\dist_a=m$. Then $\dist_b + \dist_{a\ga b}\ge n-m$. By \lref{Lm:TwoCosets}, the cosets $H\ga b$ and $H\ga (a\ga b)$ are distinct. Since $\dist_c$ is constant within every right coset of $H$ by \lref{Lm:ConstantOnCosets}, there are $2h$ elements with average value of $\dist_c$ at least $(n-m)/2$. On one of these $2$ cosets, $\dist_c\ge (n-m)/2$, which puts this coset into $G\setminus K$, as $(n-m)/2 > n/3$ (using $m<n/3$). If we temporarily assume that $n-k<2h$, the second coset cannot be located in $G\setminus K$, so we have
\begin{equation}\label{Eq:I5}
    \text{if $n-k<2h$ then }\dist(\ga,\gb)\ge h(n-m) + (n-k-h)\lceil n/3\rceil + (k-2h)m.
\end{equation}
($188$ quadruples remain, all with $n-k\ge 2h$.)

Returning to the two cosets with average value of $\dist_c$ at least $(n-m)/2$, even if both are located within $G\setminus K$, we at least have
\begin{equation}\label{Eq:I6}
    \dist(\ga,\gb)\ge h(n-m) + (n-k-2h)\lceil n/3\rceil + (k-h)m.
\end{equation}
($99$ quadruples remain.)

In the previous inequality, we have used $\dist_a>m$ on $n-k$ rows. If $m=2$, there are at most $h+2\varphi(n/2)$ rows with $\dist_a=2$, by \lref{Lm:LimitOnm2}, so there are at least $n - (h+2\varphi(n/2)) - (n-k) = k-h-2\varphi(n/2)$ rows where we used $\dist_a=2$ in \eqref{Eq:I6} but could have used $\dist_a\ge 3$. This number of rows might be negative, but we certainly have
\begin{equation}\label{Eq:I7}
    \text{if $m=2$ then }
    \dist(\ga,\gb)\ge h(n-m) + (n-k-2h)\lceil n/3\rceil + (k-h)m + k-h-2\varphi(n/2).
\end{equation}
($89$ quadruples remain.)

Finally, we eliminate the case $n=32$:

\begin{lemma}[\cite{DrDiscr}, Lemma 4.4]\label{Lm:ManyFixedPoints2} Let $G(\ga)$, $G(\gb)$ be isomorphic $2$-groups of order $n$ satisfying $\dist(\ga,\gb)<n^2/4$. Then there exists a bijection $f:G\to G$ with at least $(n/4)(3+1/\sqrt{3})$ fixed points and such that $\gb = \ga_f$.
\end{lemma}

\begin{corollary}\label{Cr:32}
Let $G(\ga)$ be a group of order $32$. Then $\cst_{\not\cong}(\ga) > \cst_{\cong}(\ga) = \cst_0(\ga) = 168$, and there is a transposition $g:G\to G$ such that $\cst(\ga) = \dist(\ga,\ga_g)$.
\end{corollary}
\begin{proof}
Let $n=32$. Recalling the results from the Introduction, we know that
$\cst_{\not\cong}(\ga) \ge n^2/4 > \cst_0(\ga) = 6\cdot 32-24 =
168$. Let $G(\gb)\cong G(\ga)$ be such that
$\cst(\ga)=\dist(\ga,\gb)$. Since $\cst(\ga)<n^2/4$,
\lref{Lm:ManyFixedPoints2} yields a bijection $f:G\to G$ with at least
$(n/4)(3+1/\sqrt{3}) > 2n/3$ fixed points. By
\pref{Pr:ManyFixedPoints}, $\dist(\ga,\gb)\ge\cst_0(\ga)$. We are done
by Theorem \ref{Th:DrEJC}.
\end{proof}

The remaining $82$ quadruples $(n,h,k,m)$ are as follows (quadruples with the same $n$, $h$, $m$ are grouped):
\begin{equation}
\begin{array}{lll}
    ( 23, 1, \{13,14,15,16,17\}, 3 ), &( 23, 1, \{16,17\}, 4 ), &( 24, 1, \{14,15,16,17,18\}, 2 ),\\
    ( 24, 1, \{15,16,17,18\}, 3 ), &( 24, 1, 18, 4 ), &( 24, 2, \{14,16,18\}, 2 ),\\
    ( 24, 2, \{16,18\}, 3 ), &( 24, 2, 18, 4 ), &( 24, 3, \{15,18\}, 2 ),\\
    ( 24, 3, 18, 3 ), &( 24, 3, 18, 4 ), &( 24, 4, 16, 2 ),\\
    ( 24, 4, 16, 3 ), &( 25, 1, \{16,17,18\}, 3 ), &( 26, 1, \{15,16,17,18,19\}, 2 ),\\
    ( 26, 1, \{17,18,19\}, 3 ), &( 26, 2, \{16,18\}, 2 ), &( 26, 2, 18, 3 ),\\
    ( 27, 1, \{17,18,19,20\}, 3 ), &( 27, 1, 20, 4 ), &( 27, 3, 18, 3 ),\\
    ( 28, 1, \{19,20,21\}, 2 ),  &(28, 1, \{20,21\}, 3 ), &( 28, 2, 20, 2 ),\\
    ( 28, 2, 20, 3 ), &( 28, 4, 20, 2 ), &( 29, 1, \{20,21\}, 3 ),\\
    ( 30, 1, \{19,20,21,22\}, 2 ), &( 30, 1, \{21,22\}, 3 ), &( 30, 2, \{20,22\}, 2 ),\\
    ( 30, 2, 22, 3 ), &( 30, 3, 21, 2 ), &( 31, 1, \{22,23\}, 3 ),\\
    ( 33, 1, 24, 3 ), &( 34, 1, \{23,24,25\}, 2 ), &( 34, 2, 24, 2 ),\\
    ( 35, 1, 26, 3 ), &( 36, 1, 27, 2 ), &( 38, 1, \{27,28\}, 2 ),\\
    ( 38, 2, 28, 2 ), &( 42, 1, 31, 2 ).&
\end{array}\label{Eq:Quadruples}
\end{equation}

\section{Special row differences}\label{Sc:m2}

\subsection{The case $m=2$}\label{Ss:m2}

In this subsection we describe an algorithm that determines all pairs of groups $G(\ga)$, $G(\gb)$ with $m(\ga,\gb)=2$.

By \pref{Pr:m2}, we can assume that $G(\gb)$ is a fixed cyclic group of even order $n$, and there is $a\in G$ such that $|a|_\gb = n$, $|a|_\ga = n/2$.

The automorphism group $\aut(C_n)$ acts transitively on the generators of $C_n$. Thus, if $b$ is a generator of $G(\gb)$, there is $f\in\aut(\gb)$ such that $f(a)=b$. By \lref{Lm:IsoMove}, we then have $\dist_a(\ga,\gb) = \dist_{f(a)}(\ga_f,\gb_f) = \dist_b(\ga_f,\gb)$ and $\dist(\ga,\gb) = \dist(\ga_f,\gb)$. We can therefore assume without loss of generality that $a$ is a fixed generator of $G(\gb)$.

The input of the algorithm is a cyclic group $G(\gb)=C_n$ and its generator $a$. To obtain $\dist_a(\gb,\ga)=2$, we must modify the row $a$ of $G(\gb)$ in two places; say there are $v\ne w$ such that $a\ga b = a\gb b$ except for $a\ga v = a\gb w$, $a\ga w = a\gb v$. Since $a\ga b$ is now determined for every $b\in G$, we can see if $|a|_\ga = n/2$, as desired. If not, we choose different $v$, $w$.

Assume now that the locations $v$, $w$ of differences in row $a$ were chosen so that $|a|_\ga=n/2$. Let $A$ be the subgroup generated by $a$ in $G(\ga)$, and let $b$ be any element of $G\setminus A$. Denote by $a^i$ the $i$th power of $a$ in $G(\ga)$. Since $G= A\cup (A\ga b) = A\cup (b\ga A)$, we must have $b\ga a = a^\alpha\ga b$ for some $1\le \alpha<n/2$, and $b\ga b = a^\beta$ for some $0\le \beta<n/2$. Once the parameters $\alpha$, $\beta$ are chosen, the operation $\ga$ is determined, namely:
\begin{align*}
    a^i\ga a^j &= a^{i+j},\\
    a^i\ga (a^j\ga b) &= a^{i+j}\ga b,\\
    (a^i\ga b)\ga a^j &= a^i\ga (b\ga a^j) = a^i\ga (a^{j\alpha}\ga b) = a^{i+j\alpha}\ga b,\\
    (a^i\ga b)\ga (a^j\ga b) &= a^i\ga (b\ga a^j)\ga b = a^{i+j\alpha}\ga b\ga b = a^{i+j\alpha+\beta},
\end{align*}
for $0\le i$, $j<n/2$. We do not claim that this operation defines a group, only that there is no alternative way to define $\circ$ that does produce a group (as it happens, the smallest distance is achieved when $\ga$ does define a group).

It therefore suffices to consider all choices of $v$, $w$, $\alpha$, $\beta$ and find the resulting groups closest to $G(\gb)$.
Both authors independently ran this algorithm and discovered that in all cases the nearest group $G(\ga)$ was isomorphic to $C_{n/2}\times C_2$ and satisfied
\[
\dist(\ga,\gb)=
\begin{cases}
n^2/4&\text{when }n\equiv0\mod4,\\
n^2/4-1&\text{when }n\equiv2\mod4.\\
\end{cases}
\]

Since $n^2/4-1>\cst_0(n)$ when $n>20$, the quadruples of
\eqref{Eq:Quadruples} with $m=2$ can therefore be eliminated. (43
quadruples remain.)

\subsection{Some cyclic cases}\label{Ss:Cyclic}

Among the remaining orders $n$ of \eqref{Eq:Quadruples}, if $n$ belongs to $\{23$, $29$, $31$, $33$, $35\}$, the only group of order $n$ is the cyclic group $C_n$. For these orders, the search therefore amounts to determination of $\dist([C_n],[C_n])$, a difficult task in general.

Let $G(\ga)$ be a cyclic group of order $n$. For any group $G(\gb)$, define
\begin{displaymath}
    m' = m'(\ga,\gb) = \textstyle\min\{\dist_a(\ga,\gb);\;|a|_\ga = n\}.
\end{displaymath}
Recall that $C_n$ has $\varphi(n)$ generators. Since $m'$ might be bigger than $m$, we can refine \eqref{Eq:I6} as follows,
\begin{equation}\label{Eq:Ig}
    \dist(\ga,\gb)\ge h(n-m) + (n-k-2h)\lceil n/3\rceil + (\varphi(n) - (n-k))m' + (n-\varphi(n)-h)m,
\end{equation}
where we first count elements in the two cosets of $H$, then all remaining elements of $G\setminus K$, then all remaining generators, and then the remaining elements in $G\setminus H$, if any.

To eliminate all remaining quadruples with $n\in\{29,31,33,35\}$ (resp.\ $n=23$), it suffices to set $m'=4$ (resp.\ $m'=5$) in \eqref{Eq:Ig}.

We are therefore interested in the following algorithm, with parameter
$d$: Given $G(\ga)\cong C_n$, find $G(\ga)\cong C_n$ closest to
$G(\gb)$ that has $\dist_a(\ga,\gb)=d$ for some generator $a$ of
$G(\ga)$.

The idea is similar to Subsection \ref{Ss:m2}, but we reverse the
roles of the groups $G(\ga)$ and $G(\gb)$. Let $a\in G$ be such that
$|a|_{\gb}=\ell$. We wish to have $|a|_{\ga}=n$ and
$\dist_a(\ga,\gb)=m'$. By \lref{Lm:OrderMismatched}, we can assume
that $n/\ell\le d$ (since $|a|_{\ga}=n$), that is, $\ell\ge n/d$.

Let us fix $a\in G$ with the above properties. We now need to make $d$
changes to row $a$ of $G(\gb)$, focusing on only those changes that
result in $|a|_{\ga}=n$. Once such a change is made, the group
$G(\ga)$ is determined.

\begin{remark}
When $n$ is a prime, the search can be sped up by taking advantage of
the automorphism group of $C_n$ (since all nonidentity elements are
generators), and by analyzing which permutations of $\diff_a(\ga,\gb)$
result in $|a|_{\ga}=n$. See \cite{VoMS} or \cite{VoProc} for details.
We did not employ these improvements here in order to keep the code
simpler.
\end{remark}

For every quadruple $(n,h,k,m)$ of \eqref{Eq:Quadruples} with
$n\in\{23,29,31,33,35\}$, the algorithm (with $d=3$ if
$n\in\{29,31,33,35\}$ and with $d\in\{3,4\}$ if $n=23$) returns
minimal distance at least as big as $\cst_0(n)$. (30 quadruples
remain.)

\section{General algorithm for $\dist([\ga],[\gb])$}\label{Sc:Algorithm}

Here is an algorithm that finds $d=\dist([\ga],[\gb])$. By \pref{Pr:IsoDist}, we have $d=\dist([\ga],\gb) = \min\{\dist(\ga_f,\gb);\;f:G\to G$ is a bijection, $G(\ga_f)\ne G(\gb)\}$.

\bigskip
\hrule
\bigskip

When $n<5$ a brute force algorithm is sufficient. Let us therefore assume that $n\ge 5$ and, by Lemma \ref{Lm:Neutral}, that $f(1)=1$ and thus $1\in H$.

Either $H=1$ or there exists a prime $p$ and a subgroup $\overline{H}\le H$ of $G(\gb)$ of order $p$. The main cycle of the algorithm proceeds over all subgroups $\overline{H}\le G(\gb)$ of prime order $p$ or $p=1$, with $|\overline{H}|$ in descending order. From now on we will write $H$ instead of $\overline{H}$, since the fact that $H$ might be larger is irrelevant in the search.

Assume that $\dist_{min}$ is the smallest distance found by the algorithm so far, and let $H\le G(\gb)$, $|H|=p$ be given. We need to consider all bijections $f:G\to G$ such that $G(\gc) = G(\ga_{f^{-1}})$ and $G(\ga)$ agree on at least $H$. The inverse $f^{-1}$, rather than $f$, is used for notational convenience, and we then have $f(a\gc b) = f(a)\ga f(b)$.

The algorithm is a depth-first search on all partially defined $1$-to-$1$ maps $f:G\to G$, where the maps are lexicographically ordered as follows: Let $\dom(f)$ denote the domain of $f$, and let $G=\{1,\dots,n\}$. Let $f$, $g:G\to G$ be two partially defined maps. Then we say that $g<f$ if and only if there exists $i\in\dom(f)$ such that (a) for every $j\le i$, if $j\in\dom(f)$ then $j\in\dom(g)$, (b) for every $j<i$, if $j\in\dom(f)$ then $g(j)=f(j)$, (c) $g(i)<f(i)$.

The search starts as follows: Let $x$ be a generator of $H$. Then $f(x)$ is an element of order $p$ in $G(\ga)$, because we demand that $x\in H(\gc,\gb) = H$ and that $f:G(\gc)\to G(\ga)$ is an isomorphism. The second cycle of the algorithm is therefore over all elements $y=f(x)$ such that $|y|_\ga=p$.

Once $f(x)$ is known, we can extend $f$ onto $H$. Indeed, we have $f(x\gb x) = f(x\gc x)$ by our assumption that $H=H(\gc,\gb)$, and $f(x\gc x) = f(x)\ga f(x)$ because $f:G(\gc)\to G(\ga)$ is a homomorphism. Similarly for higher powers of $x$.

To extend the domain of $f$ further, we systematically choose $b\not\in\dom(f)$, $c\not\in\im(f)$, and declare $f(b)=c$. Once again, we can now extend $f$ onto the coset $H\gb b$, as for $y\in H$ we must have $f(y\gb b) = f(y\gc b) = f(y)\ga f(b)$.

Anytime we extend the domain of $f$ by another coset of $H$, we can calculate the guaranteed distance between the partially defined group $G(\gc)$ and the group $G(\gb)$ by counting only those pairs $(a,b)$ that satisfy: $a\in\dom(f)$, $b\in\dom(f)$, $a\gc b\in\dom(f)$ and $f(a\gc b)\ne f(a)\ga f(b)$. If this distance exceeds $\dist_{min}$, we terminate this branch of the depth-first search.

Whenever we extend the domain of $f$ by another coset, we consider the automorphisms $g\in\aut(\ga)$ and $\ell\in \aut(\gb)$. By \lref{Lm:Aut}, $\dist(\ga_{\ell fg},\gb) = \dist(\ga_{f},\gb)$. It is also easy to see that $H(\ga_{\ell fg}, \gb) = H(\ga_f,\gb)$. Therefore, if $\ell fg<f$, we have seen $\ell fg$ before $f$ (in this cycle with the same $H$), $f$ cannot do better than $\ell fg$ as far as distance is concerned, so we terminate the branch.

If $\dom(f)=G$ anytime in the search, we calculate the full distance $\dist(\gc,\gb)$ and compare it to $\dist_{min}$.

\bigskip
\hrule
\bigskip

The following improvements make the algorithm faster.
\begin{enumerate}
\item[-] the distance $\dist(\gc,\gb)$ is calculated incrementally, in every step considering only rows, columns and values from the coset of $H$ on which $f$ has just been defined,
\item[-] the comparison of $\ell fg$ to $f$ is costly, and it is better to stop using it in the search from a certain (heuristically determined) depth in the search,
\item[-] assuming that the algorithm has gone through all values of $p>1$ and is now in the cycle $p=1$, the guaranteed distance can be calculated with a bonus. Namely, since we have $H=1$ at this stage, we can assume that every row not in the domain of $f$ contains $2$ (resp. $3$) differences when $n$ is even (resp. odd), by \lref{Lm:MinimalM}.
\end{enumerate}

The algorithm is sufficiently fast to deal with all orders $n\le 22$, albeit in some cases we merely verified that $\dist([\ga],[\gb])$ exceeds $\cst(\ga)$, without actually determining $\dist([\ga],[\gb])$. The case $n=22$ alone took more than a week of computing time. It was therefore of some importance that we could assume $G(\ga)\cong G(\gb)$ when $n=22$, by \lref{Lm:2p}.

The results of the search for $n\le 22$ are summarized in Theorem \ref{Th:Main}.

The algorithm can also be used to eliminate all remaining cases of \eqref{Eq:Quadruples} with $h>1$; we simply do not run the algorithm with any values $p$ less than $h$. This leaves us with the following twenty quadruples $(n,h,k,m)$:
\begin{equation}\label{Eq:Quadruples2}
\begin{array}{lll}
    ( 24, 1, \{15,16,17,18\}, 3 ),& (24, 1, 18, 4), &(25, 1, \{16,17,18\}, 3)\\
    ( 26, 1, \{17,18,19\}, 3 ),&( 27, 1, \{17,18,19,20\}, 3 ), & ( 27, 1, 20, 4 )\\
    ( 28, 1, \{20,21\}, 3 ),&( 30, 1, \{21,22\}, 3 ).&
\end{array}
\end{equation}
We eliminate them in Section \ref{Sc:Stubborn}, but first we need to introduce results on rainbow matchings in edge-colored graphs.

\section{Rainbow matchings and the graph $\Gamma_U$}\label{Sc:Graph}

Call an edge-colored graph \emph{restricted} if it has at most $3$
edges of any given color, and if at most two edges of the same color
are incident at any vertex. Recall that a \emph{rainbow
  $\ell$-matching} in an edge-colored graph is a set of $\ell$
disjoint edges colored by distinct colors. For $v>1$ and $\ell>0$,
define $\mu_\ell(v)$ to be the minimum number of edges a
restricted graph on $v$ vertices must have in order to guarantee a
rainbow $\ell$-matching. If there exists a coloring of the
complete graph on $v$ vertices that yields a restricted graph without
a rainbow $\ell$-matching, then we define $\mu_\ell(v)=\binom{v}{2}+1$.

\begin{proposition}\label{Pr:Graph} We have $\mu_1(v)=1$ for every $v\ge 2$, $\mu_2(v)=7$ if $4\le v\le 6$, $\mu_2(v)=v$ if $v\ge 7$, $\mu_3(6)=13$, $\mu_3(7)=15$, $\mu_3(8)=15$, $\mu_3(9)=16$ and $\mu_3(10)=18$.
\end{proposition}

We now describe the algorithm used to
establish Proposition~\ref{Pr:Graph}.  The aim was to find the
greatest number of edges that a restricted graph on $v$ vertices can
have without containing a rainbow $\ell$-matching.  We began with an
empty graph on $v$ vertices, and added the edges one color at a
time. We will refer to the process of adding all the edges of a
particular color as a {\em stage}. In each stage, we read in each of
the graphs from the previous stage, one at a time, added edges of
the new color in all possible ways, and output any graph which was
not isomorphic (by an isomorphism that respects the edge coloring,
but is allowed to permute colors) to a graph we had already seen. The
isomorphism testing was accomplished by \texttt{nauty} \cite{nauty}.

After a graph was read in stage $c$, we found all rainbow
$(\ell-1)$-matchings in it. Any edge disjoint from any such matching is
unavailable to be colored $c$. Typically this rule leaves very few
edges still available. We also sped up the search by making several
other assumptions. Firstly, since all isolated vertices are isomorphic,
vertex $j+1$ would not be connected to its first edge before vertex $j$
was. Secondly, for $c>1$ we insisted that there were not more edges of
color $c$ than there were of color $c-1$. Thirdly, we assumed that
there was at most one color which occurs on only one edge. This last
assumption is justified because if two colors each only occurred on
one edge then we could replace those two colors by a single color.
The result would still be a restricted graph, and would not have a
rainbow $\ell$-matching unless the original graph did.

As a partial validation of our computations, it is easy to confirm by
hand that the values quoted in Proposition~\ref{Pr:Graph} are lower
bounds on $\mu_\ell(v)$. First note that we can prevent a rainbow
$\ell$-matching by having no $\ell$-matchings at all. This can be achieved
by having a set of $\ell-1$ vertices that cover all edges, in which case
we can have up to ${\ell-1\choose 2}+(\ell-1)(v-\ell+1)=(\ell-1)(v-\ell/2)$ edges.
Thus $\mu_\ell(v)\ge 1+(\ell-1)(v-\ell/2)$ whenever $v\ge \ell-1$.  This
elementary lower bound is actually achieved for $\mu_1(v),\,v\ge1$;
$\mu_2(v),\,v\ge7$; and $\mu_3(v),\,v\in\{9,10\}$. To give a lower
bound for the other values quoted in Proposition~\ref{Pr:Graph}, we
display in \fref{f:rnbw} graphs with (a) 4 vertices, 6 edges and no
rainbow $2$-matching, (b)
7 vertices, 14 edges and no rainbow $3$-matching. Edge colors are
indicated by the different styles of lines. By deleting either
of the degree 2 vertices from (b) we obtain a graph with
6 vertices, 12 edges and no rainbow $3$-matching. These examples show
that $\mu_2(v)\ge7$ for $v\ge4$, $\mu_3(6)\ge13$
and $\mu_3(8)\ge\mu_3(7)\ge15$.

\begin{figure}[htb]
\[(a)\;\includegraphics[scale=0.6]{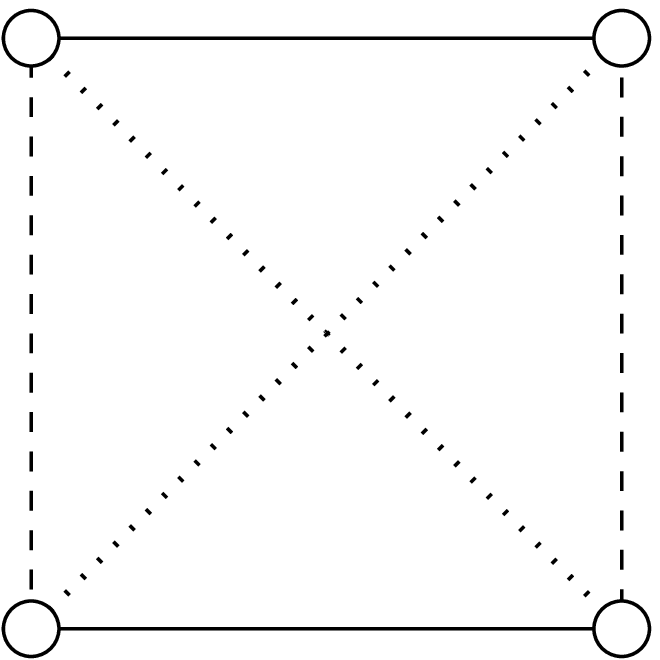}\qquad
(b)
\;\includegraphics[scale=0.6]{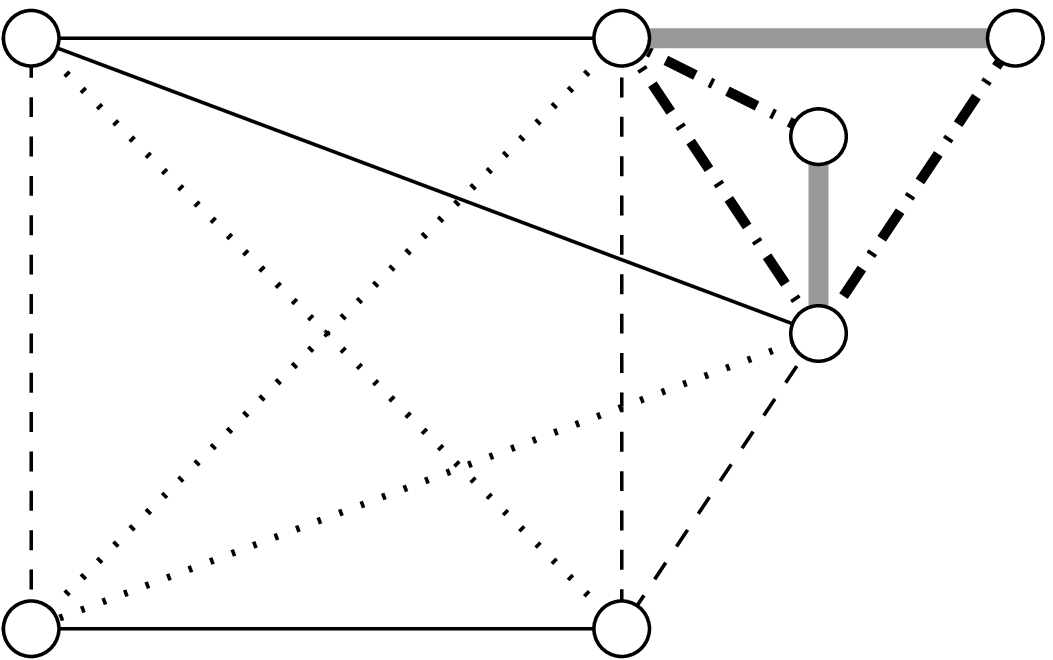}
\]
\caption{\label{f:rnbw}Restricted graphs giving lower bounds for Proposition~\ref{Pr:Graph}.}
\end{figure}

The statement in \pref{Pr:Graph} that $\mu_2(v)=v$ for $v\ge7$ is
easily seen.  We have already argued that $\mu_2(v)\ge v$.  Suppose we
have a restricted graph with $v\ge7$ vertices and $v$ edges and
no rainbow $2$-matching.  Any graph with $v>3$ vertices and $v$ edges
has a $2$-matching; in our case both edges must have
the same color $c$. Every edge of color different from $c$ must join
the two edges of the $2$-matching, and there are only 4 possible
places to put such an edge. There may be a third edge of color $c$, but
that is all. Thus our graph has at most $7$ edges. The case $v=e=7$ can
be handled by more detailed case analysis, or ruled out by our computer
programs.

Let us now return to the problem of distances of groups. The following subsets of $\diff(\ga,\gb)$ will play an important role in the analysis of the cases \eqref{Eq:Quadruples2}. Let
\begin{equation}\label{e:defRSTU}
\begin{split}
    R = R(\ga,\gb)&= \{(a,a)\in \diff(\ga,\gb);\;a\in K\},\quad r=r(\ga,\gb)=|R|,\\
    S = S(\ga,\gb)&= \{(a,b)\in \diff(\ga,\gb);\;a\in K,\,b\in K,\,a\ne b\},\quad s=s(\ga,\gb)=|S|,\\
    T = T(\ga,\gb)&= \{(a,b)\in \diff(\ga,\gb);\;a\in K,\,a\ga b\in K\},\quad t=t(\ga,\gb)=|T|,\\
    U' = U'(\ga,\gb)&= \{(a,b)\in \diff(\ga,\gb);\;a\in K,\,a\ga b\not\in K,\,b\not\in K\}.
\end{split}
\end{equation}
Note that, $R$, $S$, $T$, $U'$ are disjoint and $R\cup S\cup T\cup U'
= \diff(\ga,\gb)\cap (K\times G)$, a set that contains
at least $m\ge3$ elements in every row indexed by $K\setminus H$.  Let $U$
be any minimal subset of $U'$ subject to the condition that $R\cup
S\cup T\cup U$ contains at least $3$ elements within each row indexed
by $K\setminus H$. Let $u=u(\ga,\gb)=|U|$. We have
\begin{equation}\label{Eq:rstu}
    r+s+t+u\ge 3(k-h).
\end{equation}

Note that if $(a,b)\in S(\ga,\gb)$, then we must have $a\ga b\not\in K$
(and $a\gb b\not\in K$), since otherwise $\dist_a+\dist_b+\dist_{a\ga b} < n$
(and $\dist_a+\dist_b+\dist_{a\gb b}< n$), a contradiction of
\lref{Lm:Triple}. Similarly, if $(a,b)\in T(\ga,\gb)$ then $b\not\in
K$.

Define a multigraph $\Gamma'_U$ on vertices $V=G\setminus K$ by
declaring $\{x,y\}\subseteq V$ to be an edge if and only if $x\ne y$
and $\{x,y\}=\{b,a\ga b\}$ for some $(a,b)\in U$. Such an edge
$\{x,y\} = \{b,a\ga b\}$ will be colored $a$.

If $\{x,y\}=\{b,a\ga b\} = \{d,c\ga d\}$ is an edge of $\Gamma'_U$ for
some $(a,b)$, $(c,d)\in U$, one of the following situations occurs. If
$b=d$ then $a\ga b=c\ga b$, $a=c$, and $(a,c)=(b,d)$. Otherwise
$b=c\ga d$, $d=a\ga b$, $a\ga c\ga d=d$, and $c=a^\ga$. Therefore
$\Gamma'_U$ has at most two edges between any two given vertices. If
two distinct edges colored $a$ are incident to a vertex of
$\Gamma'_U$, they are of the form $\{b,a\ga b\}$, $\{c,a\ga c\}$ for
some $b\ne c$. Then, without loss of generality, we have $b=a\ga
c$. This means that no more than two distinct edges colored $a$ are
incident to a vertex of $\Gamma'_U$.

Let $\Gamma_U$ be the simple subgraph of $\Gamma'_U$ obtained by
suppressing any multiple edges.  By construction, $\Gamma_U$ is a
restricted graph on $n-k$ vertices.  Moreover, any edge of $\Gamma_U$
colored $a$ stems from some element $(a,b)\in U$.  Later we will use
\eqref{Eq:rstu} to find a lower bound for $u$.  In creating $\Gamma_U$
from $\Gamma'_U$, there are at least $\lceil u/2\rceil$ edges that
remain. Having built a restricted graph with at least a certain number
of edges, we will be in a position to employ \pref{Pr:Graph}.

\section{Eliminating cases with a rainbow $3$-matching in $\Gamma_U$}\label{Sc:ApplyGraph}

For the rest of this section, fix $G(\ga)$, $G(\gb)$, assume that $m(\ga,\gb)\ge 3$, let $q=\lceil n/3\rceil$, and let
\begin{displaymath}
    \pi = \dist(\ga,\gb) - \big((k-h)m + (n-k)q\big)
\end{displaymath}
be the number of differences above those guaranteed by the fundamental inequality \eqref{Eq:I4}. We will refer to $\pi$ as the \emph{profit}. If we wish to indicate the profit obtained in particular rows $r_1$, $\dots$, $r_\ell$, we use the notation $\pi(r_1,\dots,r_\ell)$.

We present a series of lemmas that eliminate most quadruples of
\eqref{Eq:Quadruples2}. While attempting to eliminate a quadruple
$(n,h,k,m)$ from \eqref{Eq:Quadruples2}, we proceed as follows: We use
Lemmas \ref{Lm:R}, \ref{Lm:LargeR} and, if $n=2p$, also \lref{Lm:2pR},
to obtain an upper bound on $r$, with default bound $r\le k-h$. Lemmas
\ref{Lm:RowS} and \ref{Lm:7S} yield an upper bound on $s$, with
default bound $s\le (k-1)(k-h)$. The dual Lemmas \ref{Lm:RowT} and
\ref{Lm:7T} yield an upper bound on $t$, with default bound
$t\le(n-k)(k-h)$. Then \eqref{Eq:rstu} provides a lower bound for $u$.
Recall that there are $n-k$ vertices and at least
$\lceil u/2\rceil$ edges in $\Gamma_U$. We then use \pref{Pr:Graph}
to determine the maximal $\ell$ such that
$\lceil u/2\rceil\ge\mu_\ell(n-k)$.  Finally, we apply
\lref{Lm:Matching}, and if this yields a sufficient profit then
$(n,k,h,m)$ is eliminated.

The challenge is not to count profit on the same row more than once. We often use the following \emph{disjunction tricks} to make sure that this does not happen. If $(a,a)\in R$ then we have $2\dist_a+\dist_{a\ga a}\ge n$ (by Lemma \ref{Lm:Triple} that we are going to use without reference) and $2\dist_a+\dist_{a\gb a}\ge n$. Thus $\pi(a\ga a)$, $\pi(a\gb a)\ge n-(q-1)$ and we are free to choose one of the two distinct rows $a\ga a$, $a\gb a$ of $G\setminus K$. If $(a,b)\in S$ then $\dist_a+\dist_b+\dist_{a\ga b}\ge n$ and $\dist_a+\dist_b+\dist_{a\gb b}\ge n$. Since $a$, $b\in K$, we must have $a\gb b\in G\setminus K$, too, $\pi(a,b,a\ga b)$, $\pi(a,b,a\gb b)\ge n-(2m+q)$, and we are free to choose one of the two distinct rows $a\ga b$, $a\gb b$ of $G\setminus K$. Finally, if $(a,b)\in T$, then again $\dist_a+\dist_b+\dist_{a\ga b}\ge n$, $\dist_a+\dist_b+\dist_{a\gb b}\ge n$, we have $a\ga b\in K$, but we might have $a\gb b\in G\setminus K$. It is therefore better to consider the element $c = a^{\gb}*(a\ga b)$ and the triple $(a,c,a\gb c)$ with respect to $G(\gb)$. Indeed, $a\in K$, $a\gb c = a\ga b\in K$, $b\ne c$ (since $a\ga b\ne a\gb b$), thus $a\ga c\ne a\ga b = a\gb a^\gb\gb(a\ga b) = a\gb c$, $c\not\in K$, and $(a,c)\in T(\ga,\gb)$. We then have $\pi(a,b,a\ga b)$, $\pi(a,c,a\gb c)\ge n-2m-q$ and we are free to choose one of the two alternatives.

\begin{lemma}\label{Lm:R}
Suppose that $(a_1,a_1)$, $\dots$, $(a_\ell,a_\ell)\in R$ are
distinct. Then $\pi\ge\ell(n-2q-m+1)$ provided that for $1\le
i\le\ell$ there is $\gc_i\in\{\ga,\gb\}$ such that $a_1\gc_1 a_1$,
$\dots$, $a_\ell\gc_\ell a_\ell$ are distinct. In particular, this
condition is always satisfied if $n$ is odd or if $\ell=2$.
\end{lemma}
\begin{proof}
For any $a$ with $(a,a)\in R$ we have
$\dist_a+\dist_a+\dist_{a\ga a}\ge n$ by \lref{Lm:Triple}.
Since $a\in K$, it follows that
$\dist_a+\dist_{a\ga a}\ge n-\dist_a\ge n-q+1$.
Since \eqref{Eq:I4} guaranteed
only $m+q$ differences on the two rows $a$, $a\ga a$, the profit on
these two rows is at least $n-q+1-(m+q) = n-2q-m+1$.
A similar argument applies to the pair of rows $a$ and $a\gb a$.

When $a_1\gc_1 a_1$, $\dots$, $a_\ell\gc_\ell a_\ell$ are distinct, we
immediately obtain $\pi\ge \ell(n-2q-m+1)$ as $a_i\gc_i a_i\in
G\setminus K$ and $a_i\in K$ for all $i$. In particular, if $n$ is odd
we can choose $\gc_i=\ga$ for all $i$, since the squaring map is a
permutation in groups of odd order.

The case $\ell=2$ is resolved by a disjunction trick, using
$a_2\ga a_2$ or $a_2\gb a_2$.
\end{proof}

\begin{lemma}\label{Lm:LargeR}
Suppose that $r\ge 4$. Then $\pi\ge \min\{2(n-q-2m)$, $3(n-2q-m+1)\}$.
\end{lemma}
\begin{proof}
First suppose that there are $(a,a)$, $(b,b)\in R$ such that $M=\{a\ga a$, $a\gb a$, $b\ga b$, $b\gb b\}$ satisfies $|M|\ge 3$. Pick any $c$ such that $a\ne c\ne b$ and $(c,c)\in R$, which is possible since $r\ge 3$. If $c\ga c\not\in M$ then  $|\{a\gc a$, $b\gd b$, $c\ga c\}|\ge 3$ for some $\gc$, $\gd\in\{\ga,\gb\}$, and \lref{Lm:R} implies $\pi\ge 3(n-2q-m+1)$. Let us therefore assume without loss of generality that $c\ga c=a\ga a$. Note that we then have $c\ga c\ne a\gb a$. If $c\ga c=b\ga b$ then $c\ga c\ne b\gb b$ and also $b\gb b\ne a\gb a$ (else $a\ga a=c\ga c=b\ga b$, $b\gb b=a\gb a$, $|M|<3$), so $a\gb a$, $c\ga c$, $b\gb b$ are distinct, and we are done by \lref{Lm:R}. If $c\ga c =b\gb b$ then $c\ga c\ne b\ga b$ and $b\ga b\ne a\gb a$ (else $b\ga b=a\gb a$, $b\gb b = c\ga c = a\ga a$, $|M|<3$), so $a\gb a$, $c\ga c$, $b\ga b$ are distinct, and we are done by \lref{Lm:R}. Thus we can assume $b\ga b\ne c\ga c\ne b\gb b$. Since either $a\gb a\ne b\ga b$ or $a\gb a\ne b\gb b$, the elements $c\ga c$, $a\gb a$, $b\gc b$ are distinct for some $\gc\in\{\ga,\gb\}$, and we finish with \lref{Lm:R} again.

We can therefore suppose that there are $x$, $y\in G$ such that $\{a\ga a,a\gb a\}=\{x,y\}$ for every $(a,a)\in R$. Let $\rho = \min\{\dist_x,\dist_y\}$. Then for every $(a,a)\in R$ we have $\dist_a\ge (n-\rho)/2$, because $\dist_a+\dist_a+\dist_{a\gc a}\ge n$ for $\gc\in\{\ga,\gb\}$, and $\dist_{a\gc a}\le \rho$ for some $\gc\in\{\ga,\gb\}$. The profit on the rows $\{a;\;(a,a)\in R\}\cup\{x,y\}$ is therefore at least $r((n-\rho)/2-m) + 2(\rho-q)$. If $(n-\rho)/2-m\ge 0$, the assumption $r\ge 4$ yields profit at least $2(n-q-2m)$. Suppose that $(n-\rho)/2-m<0$. Then $\rho>n-2m$, so $\dist_{a\ga a}$, $\dist_{a\gb a}>n-2m$ for every $(a,a)\in R$. Let $(a,a)$, $(b,b)\in R$ be distinct. Then there is $\gc\in\{\ga,\gb\}$ such that $a\ga a$, $b\gc b$ are distinct, and the profit on these rows is at least $2(n-2m-q+1)$.
\end{proof}

\begin{lemma}\label{Lm:2pR}
Suppose that $n=2p$ for some prime $p$. Then $\pi\ge \lceil r/2\rceil(n-2q-m+1)$.
\end{lemma}

\begin{proof}
The only groups of order $2p$ are the cyclic group $C_{2p}$ and the
dihedral group $D_{2p}$. In these groups, for every $a\ne 1$ there are
at most two elements $b$ such that $a=b^2$. Hence there are at least
$\ell = \lceil r/2\rceil$ distinct elements $(a_1,a_1)$, $\dots$,
$(a_\ell,a_\ell)\in R$ with $a_1\ga a_1$, $\dots$, $a_\ell\ga a_\ell$
distinct. We are done by \lref{Lm:R}.
\end{proof}

Let us now establish several results concerning an upper bound on $s$.

\begin{lemma}\label{Lm:RowS}
Let $a\in K$ and let $b_1$, $\dots$, $b_\ell\in K$ be distinct.
Suppose that either $(a,b_1)$, $\dots$, $(a,b_\ell)\in S$, or $(b_1,a)$, $\dots$, $(b_\ell,a)\in S$. Then $\pi\ge \ell(n-2q-m+1)+q-m-1$.
\end{lemma}
\begin{proof}
Assume that $(a,b_1)$, $\dots$, $(a,b_\ell)\in S$, with the transposed situation being similar. By \lref{Lm:Triple}, for every $i$ we have $\dist_{b_i}+\dist_{a\ga b_i} \ge n-\dist_a\ge n-q+1$. Since $(a,b_i)\in S$, we have $a\ne b_i$ for every $1\le i\le\ell$. Hence
the elements $a$, $b_1$, $\dots$, $b_\ell$, $a\ga b_1$, $\dots$, $a\ga b_\ell$ are distinct, with $a\ga b_i\not\in K$. The profit on $a$, $b_1$, $a\ga b_1$ is at least $n-(2m+q)$, while the profit on each of the $\ell-1$ pairs of rows $b_i$, $a\ga b_i$ for $i>1$ is at least $n-q+1-(m+q)$.
\end{proof}

\begin{lemma}\label{Lm:SquareS}
If there are $(a,b)$, $(c,d)\in S$ such that $|\{a,b,c,d\}|=4$ then
$\pi\ge 2(n-q-2m)$.
\end{lemma}
\begin{proof}
If $a\ga b\ne c\ga d$ then the profit at the distinct rows $a$, $b$,
$c$, $d$, $a\ga b$, $c\ga d$ is at least $2(n-q-2m)$, by
\lref{Lm:Triple}. Otherwise use a disjunction trick and $c\gb d$ instead of $c\ga d$.
\end{proof}

\begin{lemma}\label{Lm:7S}
If $s\ge 7$ then $\pi\ge 2(n-q-2m)$.
\end{lemma}
\begin{proof}
If there are three elements of $S$ in the same row or in the same
column, \lref{Lm:RowS} implies $\pi\ge 3(n-2q-m+1)+(q-m-1)\ge
2(n-q-2m)$. Suppose that no three elements of $S$ are in the same row
or in the same column.

Define a multigraph $\Gamma_S$ on $K$ where $\{x,y\}$ is an edge if
and only if $(x,y)\in S$ or $(y,x)\in S$. Then $\Gamma_S$ has $s$
edges, there are no more than two edges between any two vertices of
$S$, and we claim that $\Gamma_S$ has a $2$-matching.

Suppose that $\Gamma_S$ has a vertex $x$ with two distinct neighbours
$y$ and $z$. By our assumptions on $S$, there are at most $4$ edges
incident with $x$. Also, there are at most $2$ edges between $y$ and
$z$. Therefore if $s\ge7$ then there is an edge disjoint from either
$\{x,y\}$ or $\{x,z\}$, yielding the required $2$-matching.

Alternatively, if no such $x$ exists then edges are disjoint unless
they join the same pair of vertices, and it is trivial to find a
$2$-matching.

Any $2$-matching in $\Gamma_S$ yields $\pi\ge 2(n-q-2m)$ by
\lref{Lm:SquareS}.
\end{proof}

We are now going to establish results for $t$ dual to Lemmas
\ref{Lm:RowS}--\ref{Lm:7S}.

\begin{lemma}\label{Lm:RowT}
Let $a\in K$ and let $b_1$, $\dots$, $b_\ell\not\in K$ be
distinct. Suppose that either $(a,b_1)$, $\dots$, $(a,b_\ell)\in T$,
or that $(a_1,b_1)$, $\dots$, $(a_\ell,b_\ell)\in T$ for some $a_1$,
$\dots$, $a_\ell\in K$ such that $a_i\ga b_i=a$. Then $\pi\ge
\ell(n-2q-m+1)+q-m-1$.
\end{lemma}
\begin{proof}
Let $(a,b_1)$, $\dots$, $(a,b_\ell)\in T$. By \lref{Lm:Triple}, for every $i$ we have $\dist_{b_i}+\dist_{a\ga b_i} \ge n-\dist_a\ge n-q+1$. We cannot have $a=a\ga b_i$ for some $i$, else $b_i=1$, $(a,b_i)\not\in\diff(\ga,\gb)$, so $(a,b_i)\not\in T$. Hence the elements $a$, $b_1$, $\dots$, $b_\ell$, $a\ga b_1$, $\dots$, $a\ga b_\ell$ are distinct, with $a\ga b_i\in K$. The profit on $a$, $b_1$, $a\ga b_1$ is at least $n-(2m+q)$, while the profit on each of the $\ell-1$ pairs of rows $b_i$, $a\ga b_i$ for $i>1$ is at least $n-q+1-(m+q)$.

Now assume that $(a_i,b_i)\in T$, $a_i\ga b_i=a$ for some $a_i\in K$, $1\le i\le \ell$. By \lref{Lm:Triple}, for every $i$ we have $\dist_{a_i}+\dist_{b_i} \ge n-\dist_a\ge n-q+1$. We cannot have $a=a_i$ for some $i$, else $a=a_i\ga b_i=a\ga b_i$, $b_i=1$, $(a_i,b_i)\not\in T$. Hence the elements $a$, $a_1$, $\dots$, $a_\ell$, $b_1$, $\dots$, $b_\ell$ are distinct. The profit on $a_1$, $b_1$, $a=a_1\ga b_1$ is at least $n-(2m+q)$, while the profit on each of the $\ell-1$ pairs of rows $a_i$, $b_i$ for $i>1$ is at least $n-q+1-(m+q)$.
\end{proof}

\begin{lemma}\label{Lm:SquareT}
If there are $(a,b)$, $(c,d)\in T$ such that $|\{a,c,a\ga b,c\ga d\}|=4$ then $\pi\ge 2(n-q-2m)$.
\end{lemma}
\begin{proof}
If $b\ne d$ then $|\{a,b,c,d,a\ga b,c\ga d\}|=6$ and we are done by \lref{Lm:Triple}. So let us assume that $b=d$. We can apply a disjunction trick and consider $e=c^\gb\gb(c\ga b)\in G\setminus K$, obtaining $e\ne b$, $\pi(c,e,c\gb e)\ge n-(2m+q)$. By our assumption, $\{a,a\ga b\}\cap\{c,c\gb e\}=\emptyset$. We therefore have additional profit of at least $n-(2m+q)$ on the rows $a$, $b$, $a\ga b$.
\end{proof}

\begin{lemma}\label{Lm:7T}
If $t\ge 7$ then $\pi\ge 2(n-q-2m)$.
\end{lemma}
\begin{proof}
If there are three elements of $T$ in the same row or with the same
product, \lref{Lm:RowT} implies $\pi\ge 3(n-2q-m+1)+(q-m-1)\ge
2(n-q-2m)$. Suppose that no three elements of $T$ are in the same row
or have the same product.

Define a multigraph $\Gamma_T$ on $K$ where $\{x,y\}$ is an edge if
and only if there is $z$ such that either $(x,z)\in T$ and $x\ga z=y$,
or $(y,z)\in T$ and $y\ga z=x$. Then $\Gamma_T$ has $t$ edges and
there are no more than two edges between any two vertices of
$T$. Arguing as in the proof of \lref{Lm:7S}, we can show that
$\Gamma_T$ has a $2$-matching.

Hence there are $(a,b)$, $(c,d)\in T$ such that $|\{a,c,a\ga b,c\ga d\}|=4$, and we are done by \lref{Lm:SquareT}.
\end{proof}

Finally, we return to the graph $\Gamma_U$ based on the set $U$.

\begin{lemma}\label{Lm:Matching}
If $\Gamma_U$ has a rainbow $\ell$-matching then $\pi\ge\ell(n-2q-m)$.
\end{lemma}
\begin{proof}
The existence of a rainbow $\ell$-matching in $\Gamma_U$ is equivalent to the existence of $\ell$ pairwise disjoint sets $\{a_i,b_i,a\ga b_i\}$, where $(a_i,b_i)\in U$, so $a_i\in K$, $b_i$, $a\ga b_i\in G\setminus K$. The rest follows from \lref{Lm:Triple}.
\end{proof}

To illustrate the procedure outlined at the beginning of this section, let us eliminate $(n,h,k,m)=(24,1,16,3)$. Since $\cst_0(24) = 120$, $q=\lceil n/3\rceil=8$, and $(n-k)q+(k-h)m=109$, we need a profit of at least $12$. \lref{Lm:R} with $r=2$ (thus $\ell=2$) yields precisely $\pi\ge 12$. We can therefore assume $r\le 1$, which \lref{Lm:LargeR} cannot improve. \lref{Lm:RowS} yields a sufficient $\pi\ge 16$ with $\ell=2$ (but $\ell=1$ does not suffice), so $s\le1(k-h)=15$. Since \lref{Lm:7S} yields $\pi\ge 20$, we can improve the bound to $s\le 6$. Similarly, \lref{Lm:RowT} with $\ell=2$ yields $t\le15$, which \lref{Lm:7T} improves with $\pi\ge 20$ to $t\le 6$. Then \eqref{Eq:rstu} allows us to assume that $u\ge3(k-h)-1-6-6=32$, and thus that $\Gamma_U$ has at least $\lceil 32/2\rceil = 16$ edges. Since $\mu_3(n-k)=\mu_3(8)=15$ by \pref{Pr:Graph}, $\Gamma_U$ contains a rainbow $3$-matching. Then $\pi\ge 3(n-2q-m)=15>12$ by \lref{Lm:Matching}, which is what we need, and $(24,1,16,3)$ is eliminated.

A straightforward calculation shows that the only remaining cases of \eqref{Eq:Quadruples} are
\begin{equation}\label{Eq:Quadruples3}
    (24,1,\{17,18\},3),\quad(25,1,\{17,18\},3),\quad (26,1,19,3),\quad (27,1,\{19,20\},3).
\end{equation}
For these surviving cases the above procedure at least yields upper bounds on $r$, $s$, $t$ and a lower bound on $u$ as follows:
\begin{align*}
    &(24,1,17,3):\quad r\le 3,\,s\le 6,\,t\le 6,\,u\ge 33,\\
    &(24,1,18,3):\quad r\le 17,\,s\le 34,\,t\le 34,\,u\ge 0,\\
    &(25,1,17,3):\quad r\le 2,\,s\le 6,\,t\le 6,\,u\ge 34,\\
    &(25,1,18,3):\quad r\le 3,\,s\le 6,\,t\le 6,\,u\ge 36,\\
    &(26,1,19,3):\quad r\le 6,\,s\le 6,\,t\le 6,\,u\ge 36,\\
    &(27,1,19,3):\quad r\le 2,\,s\le 6,\,t\le 6,\,u\ge 40,\\
    &(27,1,20,3):\quad r\le 3,\,s\le 38,\,t\le 38,\,u\ge 0.
\end{align*}

\section{Stubborn cases}\label{Sc:Stubborn}

It is easy to check that the profit obtained from a rainbow $3$-matching in $U$ is not sufficient to eliminate any of the cases \eqref{Eq:Quadruples3}. We will need more delicate profits, for instance obtained from a rainbow $2$-matching in $U$ and an element $(a,b)\in S$ such that $a$, $b$, $a\ga b$ are disjoint from the vertices and colors of the rainbow $2$-matching. We start with two dual lemmas that in certain circumstances provide upper bounds on $s$ and $t$.

\begin{lemma}\label{Lm:s3}
If $s\ge 3$ and $q\ge m+1$ then $\pi\ge 2n-3q-3m+1$.
\end{lemma}
\begin{proof}
If there are $(a,b)$, $(c,d)\in S$ with $|\{a,b,c,d\}|=4$, we are done by \lref{Lm:SquareS} and $q\ge m+1$. Otherwise there are $(a,b)$, $(c,d)\in S$ with $|\{a,b,c,d\}|=3$. If either $a=c$ and $b\ne d$, or $a\ne c$ and $b=d$, then $\pi\ge 2n-3q-3m+1$ by \lref{Lm:RowS} with $\ell=2$. The cases when $a=d$ or $b=c$ yield the same profit by an argument similar to \lref{Lm:RowS}. We cannot have $a=b$ or $c=d$ by the definition of $S$.
\end{proof}

\begin{lemma}\label{Lm:t3}
If $t\ge 3$ and $q\ge m+1$ then $\pi\ge 2n-3q-3m+1$.
\end{lemma}
\begin{proof}
If there are $(a,b)$, $(c,d)\in T$ with $|\{a,c,a\ga b, c\ga d\}|=4$, we are done by \lref{Lm:SquareT} and $q\ge m+1$. Otherwise there are $(a,b)$, $(c,d)\in T$ with $|\{a,c,a\ga b,c\ga d\}|=3$. The cases when $a=c$ or $a\ga b=c\ga d$ are handled by \lref{Lm:RowT}.

If either $a=c\ga d$ or $c=a\ga b$, we can assume without loss of generality that $a = c\ga d$. If $b\ne d$ then $a$, $b$, $c$, $d$, $a\ga b$ are distinct, and the profit on the rows $a$, $b$, $a\ga b$ is at least $n-(2m+q)$. Since $\dist_c+\dist_d\ge n-\dist_{c\ga d}\ge n-q+1$, the profit on the rows $c$, $d$ is at least $n-2q-m+1$, and the total profit is at least $2n-3q-3m+1$.

Finally suppose that $a=c\ga d$, $b=d$, and the elements $a$, $b$, $c$, $a\ga b$ are distinct. Using a disjunction trick for $(a,b)$, let us consider $(a,e=a^\gb\gb(a\ga b))\in T$ and $(c,b=d)\in T$, focusing on the rows $a$, $e$, $a\gb e = a\ga b$, $c$, $b=d$, $c\ga d$, which are distinct, except that $a=c\ga d$. We finish as above.
\end{proof}

\begin{lemma}\label{Lm:U}
We have $u\le (n-k)(n-k-1)$.
\end{lemma}
\begin{proof}
An element $(c,d)\in U$ determines the ordered pair $(d,c\ga d)\in (G\setminus K)\times (G\setminus K)$ with $d\ne c\ga d$ (since $c\ne 1$) and vice versa.
\end{proof}

We now elaborate on the idea of rainbow matchings in $U$ disjoint from
elements of $R$, $S$ and/or $T$.

For $(a,b)\in R\cup S\cup T$, let $U\swr (a,b)=\{(c,d)\in U;\;\{c,d,c\ga d\}\cap \{a,b,a\ga b\}=\emptyset\}$. For $(a,b)$, $(c,d)\in R\cup S\cup T$, let $U\swr(a,b)(c,d) = \{(e,f)\in U;\;\{a,b,a\ga b,c,d,c\ga d\}\cap \{e,f,e\ga f\}=\emptyset\}$.

\begin{lemma}\label{Lm:Uwr}
For $(a,b)\in R\cup S\cup T$, we have
\begin{displaymath}
    |U\swr(a,b)|\ge
    \left\{\begin{array}{ll}
        u - (2n-2k+1),&\text{ if }(a,b)\in R,\\
        u - (2n-2k+4),&\text{ if }(a,b)\in S\cup T.
    \end{array}\right.
\end{displaymath}
\end{lemma}
\begin{proof}
Assume that $(a,b)\in S$. Then an element $(c,d)\in U$ does not belong to $U\swr(a,b)$ if and only if one of the following occurs: $c=a$, $c=b$, $d=a\ga b$, $c\ga d=a\ga b$. Now, $c=a$ can occur for at most $2$ elements of $U$, by the definition of $U$, given that row $a$ contains $(a,b)\in S$. We have $c=b$ at most $3$ times. We have $d=a\ga b$ at most $n-k$ times, because the column $a\ga b$ contains at most $n-k$ values from $G\setminus K$. Finally, $c\ga d=a\ga b$ occurs at most another $n-k-1$ times, because the value $a\ga b$ can occur at most once in every column of $G\setminus K$, and we have already accounted for all elements of $U$ in column $a\ga b$. The result for $(a,b)\in S$ follows.

Assume that $(a,b)\in T$. Then an element $(c,d)\in U$ does not belong to $U\swr(a,b)$ if and only if one of the following occurs: $c=a$, $c=a\ga b$, $d=b$, $c\ga d=b$. The rest is analogous to the case $(a,b)\in S$.

Assume that $(a,b)=(a,a)\in R$. Then an element $(c,d)\in U$ does not belong to $U\swr(a,b)$ if and only if one of the following occurs: $c=a$, $d=a\ga a$, $c\ga d = a\ga a$. The rest is analogous to the case $(a,b)\in S$.
\end{proof}

\begin{lemma}\label{Lm:Uwr2}
If $(a,b)$, $(c,d)\in S\cup T$ then $|U\swr(a,b)(c,d)| \ge u - (4n-4k+8)$. If $(a,b)$, $(c,d)\in S$ and $|\{a,b,c,d\}|=3$ then $|U\swr(a,b)(c,d)| \ge u - (4n-4k+5)$.
\end{lemma}
\begin{proof}
For $(a,b)$, $(c,d)\in S\cup T$, apply a variation of \lref{Lm:Uwr} twice. The worst case estimate $|U\swr(a,b)(c,d)| \ge u - (4n-4k+8)$ is obtained when $|\{a,b,a\ga b,c,d,c\ga d\}|=6$.

Suppose that $(a,b)$, $(c,d)\in S$ and $|\{a,b,c,d\}|=3$. An element
$(e,f)\in U$ does not belong to $U\swr(a,b)(c,d)$ if and only if one
of the following occurs: $e\in\{a,b,c,d\}$, $f\in\{a\ga b,c\ga d\}$,
or $e\ga f\in\{a\ga b,c\ga d\}$. Since
$|\{a,b,c,d\}|=3$, we can assume without loss of generality that
either $a=c$, $b$, $d$ are distinct, or $a=d$, $b$, $c$ are
distinct. (Note that $a=b$ is impossible since $(a,b)\in S$.) If
$a=c$, $b$, $d$ are distinct, then $e=a$ occurs at most once (since
$(a,b)$, $(c,d)\in S$), $e=b$ at most $3$ times, and $e=d$ at most $3$
times. If $a=d$, $b$, $c$ are distinct, then $e=a$ occurs at most
twice, $e=b$ at most $3$ times, and $e=c$ at most twice. Hence in both
cases, $e\in\{a,b,c,d\}$ occurs for at most $7$ elements $(e,f)\in U$.

As before, we eliminate up to $2(n-k)$ elements $(e,f)\in U$ with
$f\in\{a\ga b,c\ga d\}$, and a further $2(n-k-1)$ with
$e\ga f\in\{a\ga b,c\ga d\}$.
\end{proof}

Note that in all cases \eqref{Eq:Quadruples3} we have $k>2n/3$. The following lemma will therefore apply to these cases.

\begin{lemma}\label{Lm:rs0}
Assume that $n\ge 12$ and $k>2n/3$. Then $r+s>0$ or $G(\ga)$, $G(\gb)$ are isomorphic via a transposition.
\end{lemma}
\begin{proof}
Assume that $r+s=0$. The proof of \cite[Proposition 3.1]{DrEJC} (our \pref{Pr:3n4}) goes through with $k>2n/3$ (rather than $k>3n/4$), except for part (iv), as explicitly noted already by Dr\'apal in \cite{DrEJC}. With our assumption $r+s=0$, we can replace the proof of (iv) with the following: Let $g\in G$. Then there are $a$, $b\in K$ such that $g=a\ga b$, since $k>n/2$. Assume $g=a_i\ga b_i$ for some $a_i$, $b_i\in K$, $1\le i\le 2$. If $a_1\gb b_1\ne a_2\gb b_2$ then there is $i$ such that $a_i\ga b_i\ne a_i\gb b_i$, and for this $i$ we have $(a_i,b_i)\in R\cup S$, a contradiction. Thus $a_1\gb b_1 = a_2\gb b_2$.

We can now conclude from \cite[Proposition 3.1]{DrEJC} that there is an isomorphism $f:G(\ga)\to G(\gb)$ such that $f(a)=a$ for every $a\in K$. Then by \cite[Proposition 6.1]{DrEJC}, $\dist(\ga,\gb)\ge \cst_0(\ga)$, and if equality holds, $f$ must be a transposition.
\end{proof}

The following example shows that \lref{Lm:rs0} is best possible. Let $\ga,\gb$
be defined by
\[
\begin{array}{c|ccccccccc}
\ga& 1& 2& 3& 4& 5& 6& 7& 8& 9\\
\hline
1& 1& 2& 3& 4& 5& 6& 7& 8& 9\\
2& 2& 3& 1& 5& 6& 4& 8& 9& 7\\
3& 3& 1& 2& 6& 4& 5& 9& 7& 8\\
4& 4& 5& 6& 7& 8& 9& \mk1& \mk2& \mk3\\
5& 5& 6& 4& 8& 9& 7& \mk2& \mk3& \mk1\\
6& 6& 4& 5& 9& 7& 8& \mk3& \mk1& \mk2\\
7& 7& 8& 9& \mk1& \mk2& \mk3& \mk4& \mk5& \mk6\\
8& 8& 9& 7& \mk2& \mk3& \mk1& \mk5& \mk6& \mk4\\
9& 9& 7& 8& \mk3& \mk1& \mk2& \mk6& \mk4& \mk5
\end{array}
\qquad
\begin{array}{c|ccccccccc}
\gb& 1& 2& 3& 4& 5& 6& 7& 8& 9\\
\hline
1& 1& 2& 3& 4& 5& 6& 7& 8& 9\\
2& 2& 3& 1& 5& 6& 4& 8& 9& 7\\
3& 3& 1& 2& 6& 4& 5& 9& 7& 8\\
4& 4& 5& 6& 7& 8& 9& \mk2& \mk3& \mk1\\
5& 5& 6& 4& 8& 9& 7& \mk3& \mk1& \mk2\\
6& 6& 4& 5& 9& 7& 8& \mk1& \mk2& \mk3\\
7& 7& 8& 9& \mk2& \mk3& \mk1& \mk5& \mk6& \mk4\\
8& 8& 9& 7& \mk3& \mk1& \mk2& \mk6& \mk4& \mk5\\
9& 9& 7& 8& \mk1& \mk2& \mk3& \mk4& \mk5& \mk6
\end{array}
\]
where the differences are shaded. Then $k=2n/3$ and yet the groups
are not isomorphic; $G(\ga)\cong (C_3)^2$ and $G(\gb)\cong C_9$.
By taking direct products of these two groups with other groups we
can make arbitrarily large non-isomorphic pairs where $k=2n/3$
and $r=s=0$.

\begin{lemma}\label{Lm:Profit1}
Suppose that $k=n-q+2$ and $x$, $y\in G\setminus K$, $x\ne y$. Then
there is $(v,w)\in \diff(\ga,\gb)$ such that $\{v,w,v\ga w\}\cap
\{x,y\}=\emptyset$, $v\in G\setminus K$, and either $w\in K$ or $v\ga
w\in K$.
\end{lemma}
\begin{proof}
The set $L=G\setminus(K\cup\{x,y\})$ is not closed under $\ga$ since
it does not contain $1$, so there are $v$, $w\in L$ such that $v\ga
w\not\in L$. If $v\ga w\in K$, we are done. Otherwise $v\ga w\in
\{x,y\}$, and we can assume without loss of generality that $v\ga w=x$.
Since $v\in G\setminus K$, $\dist_v\ge q = n - k + 2$, but
$|(G\setminus K)\cup\{v^\ga\ga x,v^\ga\ga y\}|\le n-k+1$ (as $v^\ga\ga
x = v^\ga \ga v\ga w = w\in G\setminus K$), so there is $z\in K$ with
$(v,z)\in\diff(\ga,\gb)$, and $v\ga z\not\in\{x,y\}$. Then $\{v,z,v\ga
z\}\cap \{x,y\}=\emptyset$, $v\in G\setminus K$, $z\in K$, and $(v,z)$
does the job.
\end{proof}

We now eliminate all the quadruples of \eqref{Eq:Quadruples3}, sorting them according to the difference $n-k$.

\emph{Case $(n,h,k,m) = (25,1,17,3)$.} To eliminate this case, we need a profit of at least $\cst_0(n)-(n-k)q -(k-h)m + 1 = 13$, and we can assume $r\le 2$, $s\le 6$, $u\ge 34$. If $s>0$ and $(a,b)\in S$ then $|U\swr(a,b)|\ge u - (2n-2k+4) \ge 14$ by \lref{Lm:Uwr}, so there is $(c,d)\in U$ such that $a$, $b$, $a\ga b$, $c$, $d$, $c\ga d$ are distinct, yielding the profit of at least $(n-q-2m) + (n-2q-m)=14>13$.
We can therefore assume that $s=0$ and $u\ge 40$. By \lref{Lm:rs0}, $r>0$ and there is $(a,a)\in R$. Then $|U\swr(a,a)|\ge u-(2n-2k+1) \ge 23$ by \lref{Lm:Uwr}. Since $\mu_2(n-k)=\mu_2(8)=8\le \lceil 23/2\rceil$, there is a rainbow $2$-matching in $U$ disjoint from $\{a,a\ga a\}$, and we obtain a sufficient profit of at least $(n-2q-m+1) + 2(n-2q-m) = 13$.

\emph{Case $(n,h,k,m) = (27,1,19,3)$.} We need a profit of at least $19$, and we can assume $r\le 2$, $s\le 6$, $u\ge 40$. If $s>0$ and $(a,b)\in S$ then $|U\swr(a,b)|\ge u-(2n-2k+4)\ge 20$ by \lref{Lm:Uwr}, $\mu_2(n-k)=\mu_2(8)=8\le \lceil 20/2\rceil$, so there is a rainbow $2$-matching disjoint from $\{a,b,a\ga b\}$, yielding a sufficient profit of $(n-q-2m)+2(n-2q-m) = 24$. We can therefore assume that $s=0$ and $u\ge 46$. By \lref{Lm:rs0}, $r>0$ and there is $(a,a)\in R$. Then $|U\swr(a,a)|\ge u-(2n-2k+1)\ge 29$ by \lref{Lm:Uwr}. Since $\mu_2(n-k)=8\le \lceil 29/2\rceil$, there is a rainbow $2$-matching disjoint from $\{a,a\ga a\}$, and we obtain a sufficient profit of at least $(n-2q-m+1)+2(n-2q-m) = 19$.

\emph{Case $(n,h,k,m) = (24,1,17,3)$.} We need a profit of at least $17$, and we can assume $r\le 3$, $s\le 6$, $t\le 6$, $u\ge 33$. If $s>0$ and $(a,b)\in S$ then $|U\swr(a,b)|\ge 15$ by \lref{Lm:Uwr}, $\mu_2(n-k) = \mu_2(7) = 7 \le \lceil 15/2\rceil$, so there is a rainbow $2$-matching in $U$ disjoint from $\{a,b,a\ga b\}$, for a sufficient profit of at least $(n-q-2m)+2(n-2q-m)=20$. Similarly if $t>0$. We can therefore assume that $s=0$, $t=0$ and $u\ge 45$. There is $(a,a)\in R$ by \lref{Lm:rs0}, $|U\swr(a,a)| \ge 30$ by \lref{Lm:Uwr}, $\mu_3(n-k) = \mu_3(7) = 15 = \lceil 30/2\rceil$, so there is a rainbow $3$-matching in $U$ disjoint from $\{a,a\ga a\}$, giving a sufficient profit of at least $(n-2q-m+1)+3(n-2q-m) = 21$.

\emph{Case $(n,h,k,m) = (26,1,19,3)$.} We need a profit of at least $20$, and we can assume $r\le 6$, $s\le 6$, $t\le 6$, $u\ge 36$. If $s>0$ and $(a,b)\in S$ then $|U\swr(a,b)|\ge 18$ by \lref{Lm:Uwr}, $\mu_2(n-k) = \mu_2(7) = 7 \le \lceil 18/2\rceil$, so there is a rainbow $2$-matching in $U$ disjoint from $\{a,b,a\ga b\}$, for a sufficient profit of at least $(n-q-2m)+2(n-2q-m)=21$. Similarly if $t>0$. If $s=0=t$ then $u\ge 52$, a contradiction of \lref{Lm:U}, which yields $u\le 42$.

\emph{Case $(n,h,k,m) = (25,1,18,3)$.} We need a profit of at least $19$, and we can assume $r\le 3$, $s\le 6$, $t\le 6$, $u\ge 36$. Suppose that $s\ge 3$. If there are $(a,b)$, $(c,d)\in S$ such that $|\{a,b,c,d\}|=4$ then \lref{Lm:SquareS} yields a sufficient profit of at least $2(n-q-2m)=20$. Otherwise, as in the proof of \lref{Lm:s3}, there are $(a,b)$, $(c,d)\in S$ such that $|\{a,b,c,d\}|=3$ and $\pi(a,b,c,d,a\ga b,c\ga d)\ge 2n-3q-3m+1 = 15$. Moreover, \lref{Lm:Uwr2} implies that $|U\swr(a,b)(c,d)|\ge 3$, so there is $(e,f)\in U$ such that $\{e,f,e\ga f\}\cap \{a,b,c,d,a\ga b,c\ga d\}=\emptyset$. Since $\pi(e,f,e\ga f)\ge n-2q-m = 4$, we have $\pi\ge 15+4=19$, as desired. We can therefore assume that $s\le 2$ and $u\ge 40$. Using \lref{Lm:Uwr2} once more, we may now deduce that $t\le2$. Hence $u\ge 44$, contradicting $u\le 42$ from \lref{Lm:U}.


\emph{Case $(n,h,k,m) = (27,1,20,3)$.} We need a profit of at least $25$, and we can assume $r\le 3$. Suppose that $s\ge 7$. Then by \lref{Lm:7S}, there are $(a,b)$, $(c,d)\in S$ such that $\pi(a,b,c,d,a\ga b,c\ga d)\ge 2(n-q-2m) = 24$. Using $(x,y) = (a\ga b,c\ga d)$ in \lref{Lm:Profit1}, we obtain $(v,w)\in\diff(\ga,\gb)$ such that $\{v,w,v\ga w\}\cap \{x,y\}=\emptyset$, $v\in G\setminus H$, and either $w\in K$ or $v\ga w\in K$. We have not yet used any of the rows $v$, $w$, $v\ga w$ that happen to be in $G\setminus K$ in our calculation of the profit. We have therefore counted at most $q + q + (q-1) = 3q-1$ differences on the rows $v$, $w$, $v\ga w$ so far, however, we have $\dist_v+\dist_w+\dist_{v\ga w}\ge n =3q$ because $(v,w)\in\diff(\ga,\gb)$. We can now increase the profit of $24$ by $1$, and we are done. Similarly, if $t\ge 7$, there are $(a,b)$, $(c,d)\in T$ such that $\pi(a,b,c,d,a\ga b,c\ga d)\ge24$ by \lref{Lm:RowT}, and we can apply \lref{Lm:Profit1} with $(x,y)=(b,d)$ to increase the profit by $1$. We can therefore assume $s\le 6$, $t\le 6$ and $u\ge 42$.
If $s\ge 3$, there are $(a,b)$, $(c,d)\in S$ with $\pi(a,b,c,d,a\ga b,c\ga d)\ge 2n-3q-3m+1 = 19$ by \lref{Lm:s3}, $|U\swr(a,b)(c,d)|\ge 6$ by \lref{Lm:Uwr2}, $(e,f)\in U$ with $\{e,f,e\ga f\}\cap \{a,b,c,d,a\ga b,c\ga d\}=\emptyset$, and $\pi(e,f,e\ga f)\ge n-2q-m = 6$, for a sufficient profit of $19+6=25$. We can therefore assume $s\le 2$ and $u\ge 46$, contradicting $u\le 42$ from \lref{Lm:U}.

\emph{Case $(n,h,k,m) = (24,1,18,3)$}. We need a profit of at least $22$.

Define $\lambda$ to be the maximum integer for which there exist distinct
$x,y\in G$ such that $\dist_x\ge\dist_y\ge\lambda$. Suppose that $\lambda\ge17$.
By \lref{Lm:Profit1} there is $(v,w)\in\diff(\ga,\gb)$ with
$\{v,w,v\ga w\}\cap\{x,y\}=\emptyset$ and $|K\cap\{w,v\ga w\}|\ge1$
so $\pi(v,w,v\ga w,x,y)\ge n-2q-m+2(\lambda-q)\ge23$. Thus we may assume that
$\lambda\le16$.

Let $\Omega$ be a maximal subset of $R\cup S\cup T$ under the constraint that
there should be a maximum of 3 elements of $\Omega$ within any row. Let
$\Sigma$ be the sum over $\Omega$ of $\dist_a+\dist_b-2m$ for elements
$(a,b)\in R\cup S$, and $\dist_a+\dist_{a\ga b}-2m$ for $(a,b)\in T$.

We claim that $\Sigma\ge|\Omega|(n-2m-\lambda)$. Each $(a,b)\in R\cup S$
satisfies
$\dist_a+\dist_b\ge n-\min\{\dist_{a\ga b},\dist_{a\gb b}\}\ge n-\lambda$.
So it suffices to show that each $(a,b)\in T$ satisfies
$\dist_a+\dist_{a\ga b}\ge n-\lambda$. Since $(a,b)\in T$, we have
$\dist_a+\dist_{a\ga b}\ge n-\dist_b$. By a disjunction trick, $(a,c)\in T$ where $c=a^\gb\gb(a\ga b)$, so
$\dist_a+\dist_{a\ga b}\ge n-\dist_c$.
Since $b$, $c$ are distinct elements of $G\setminus K$, we have
$\lambda\le\min\{\dist_b,\dist_c\}$, from which the claim follows.

Next we claim that $\Sigma\le8(37-2\lambda)$. Consider $a\in K\setminus H$.
By construction, $a$ is a row coordinate for at most 3 cells in $\Omega$.
By \lref{Lm:RowS}, there are at most 2 cells in $S$ for which $a$ is
the column coordinate, otherwise we realize a sufficient profit
of $3(n-2q-m+1)+q-m+1=24$. Similarly, using \lref{Lm:RowT}, there
are at most 2 cells $(c,d)$ in $T$ for which $a=c\ga d$. It is also
possible that $a$ is the column coordinate for a single cell in $R$.
It follows that $\Sigma\le 8\Sigma'$, where $\Sigma'$ is the sum over
$a\in K\setminus H$ of $\dist_a-m$. As the profit from $K\cup\{x,y\}$
is at least $\Sigma'+2(\lambda-q)$ we are done unless
$\Sigma'\le 21+2q-2\lambda=37-2\lambda$. This proves the claim.

Combining the previous two claims we find that
$|\Omega|\le8(37-2\lambda)/(n-2m-\lambda)=16+8/(18-\lambda)\le20$, since
$\lambda\le16$. As $\Omega\cup U$ contains three differences in every row
indexed by $K\setminus H$, it follows that $u\ge3(k-h)-|\Omega|\ge31$.
This contradicts \lref{Lm:U}, finishing the last case.

\section{Constructions}\label{Sc:Constructions}

We have now established all distances mentioned in Theorem \ref{Th:Main}. It remains to present the constructions that realize the minimal distances $\cst(\ga)=\dist(\ga,\gb)$ in situations when $\cst(\ga)<\cst_0(\ga)$.

\subsection{Cyclic and dihedral constructions}\label{Ss:Quarter}

The following two constructions \eqref{Eq:CyclicConstruction} and \eqref{Eq:DihedralConstruction} were introduced in \cite{DrConstr1}. Given a certain group $G(\ga)$ of even order $n$, they produce a group $G(\gb)$ at distance $n^2/4$ from $G(\ga)$.

Recall the graphs $\mathcal G(n)$ and $\mathcal G'(n)$ from the Introduction. It turns out that whenever two groups $G(\ga)$, $G(\gb)$ of order $n=8$ or $n=16$ are at distance $n^2/4$, there is a group $G(\gc)$ obtained from $G(\ga)$ by one of the two constructions and such that $G(\gb)\cong G(\gc)$. This follows from the fact that the graph $\mathcal G(8)$ (calculated in \cite{VoMS} and independently here) coincides with $\mathcal G'(8)$, and from the fact that the graph $\mathcal G(16)$ (calculated here for the first time) coincides with $\mathcal G'(16)$ (calculated by B\'alek \cite{Ba} and independently here).

For a fixed positive integer $m$ and the set $M = \{-m+1$, $-m+2$,
$\dots$, $m-1$, $m\}$, define $\sigma:\mathbb{Z}\to \{-1$, $0$, $1\}$ by
\begin{equation*}
    \sigma(i) = \left\{\begin{array}{ll}
        1,&i>m,\\
        0,&i\in M,\\
        -1,&i<1-m.
    \end{array}\right.
\end{equation*}

\emph{The cyclic construction.}
Let $G(\ga)$ be a group of order $n$, $S\unlhd G$, $G/S=\langle \alpha \rangle$ a cyclic group of order $2m$ and $1\ne h\in S\cap Z(G)$. Then $G(\ga)$ is the disjoint union $\bigcup_{i\in M}\alpha^i$, and we can define a new multiplication $\gb$ on $G$ by
\begin{equation}\label{Eq:CyclicConstruction}
    x\gb y = x\ga y\ga h^{\sigma(i+j)},
\end{equation}
where $x\in\alpha^i$, $y\in\alpha^j$, and $i$, $j\in M$. Then $G(\gb)$ is a group and $\dist(\ga,\gb)=n^2/4$.

\emph{The dihedral construction.}
Let $G(\ga)$ be a group of order $n$, $S\unlhd G$, $G/S$ a dihedral group of order $4m$ (where we allow $m=1$), and $\beta$, $\gamma$ involutions of $G/S$ such that $\alpha=\beta\gamma$ is of order $2m$. Let $G_0=\bigcup_{i\in M}\alpha^i$ and $G_1=G\setminus G_0$. Let $1\ne h\in S\cap Z(G_0)$ be such that $hxh=x$ for some (and hence every) $x\in G_1$. Then there are $e\in\beta$ and $f\in\gamma$ so that $G$ is the disjoint union $\bigcup_{i\in M}(\alpha^i\cup e\alpha^i)$ or $\bigcup_{j\in M}(\alpha^j\cup \alpha^jf)$, and we can define a new multiplication $*$ on $G$ by
\begin{equation}\label{Eq:DihedralConstruction}
    x\gb y=x\ga y\ga h^{(-1)^r\sigma(i+j)},
\end{equation}
where $x\in\alpha^i\cup e\alpha^i$, $y\in(\alpha^j\cup\alpha^jf)\cap G_r$, $i$, $j\in M$, and $r\in\{0,\,1\}$. Then $G(\gb)$ is a group and $\dist(\ga,\gb)=n^2/4$.

\subsection{Other constructions}\label{Ss:Other}

The following three constructions furnish the distances of
\tref{Th:Main} with $\dist(\ga,\gb)<\cst_0(\ga)$ and $n\ne 2^k$.

\emph{Construction 1.} Suppose $n\equiv 2\bmod4$ and $n\ge6$.  Let $O$ be an abelian group of order $n/2$. We have two groups defined on the set $O\times C_2$, namely $D(O)$ and the usual direct product on $O\times C_2$. The distance between these two groups is $n(n-2)/2$. When $n\in\{6,10\}$, this is $\cst(D_n)$ so $\cnb(D_n)$ contains a group isomorphic to $C_{n/2}\times C_2\cong C_n$ (although $\cnb(D_{10})$ also contains a group isomorphic to $D_{10}$, because $n(n-2)/2=6n-20=\delta_0(D_{10})$).

\def\grone{\odot} \def\grtwo{\circledast}

\emph{Construction 2.} We construct two abelian group operations $\grone$, $\grtwo$ on the set $C_a\times C_b$ where $a$ is odd.
\[
(s,t)\grone(u,v)=
\begin{cases}
(s+u,t+v+1)&\text{if $s+u\ge a$}\\
(s+u,t+v)&\text{otherwise}.
\end{cases}
\]
Clearly $\grone$ is isomorphic to $C_{ab}$ by the map $(s,t)\mapsto s+at$.

To form $\grtwo$ we take the usual group on $C_a\times C_b$ and apply the isomorphism
\[
(s,t)\mapsto\begin{cases}
(s,t+1)&\text{if $s\ge\frac12(a+1)$}\\
(s,t)&\text{otherwise}.
\end{cases}
\]
It is routine to check that $d(\grone,\grtwo)=n^2(1-a^{-2})/4$. In particular, $d(\grone,\grtwo)=2n^2/9$ when $a=3$, the nearest (proportional) distance between non-isomorphic groups \cite{ILY}. Note that $2n^2/9<\delta_0(n)$ for $n\le 21$, and indeed $\cst(C_{3b})=2n^2/9$ for $2\le b\le 7$. The above construction proves this for $b\in\{2,4,5,7\}$.

Construction 2 shows directly that the following achieve $2n^2/9$:
\begin{align*}
&\dist(C_6,C_6),
\dist(C_9,C_3^2),
\dist(C_{12},C_{12}),
\dist(C_{15},C_{15}),\\
&
\dist(C_{18},C_6\times C_3),
\dist(C_{21},C_{21}).
\end{align*}
Taking appropriate extensions of the example that realises $\dist(C_6, C_6)$,
we can show that $2n^2/9$ is also achieved in these cases:
\begin{align*}
&\dist(C_6\times C_2, C_6\times C_2),
\dist(C_6\times C_3, C_6\times C_3),
\dist(D_{12},D_{12}),\dist(\Dic_{12},\Dic_{12}).
\end{align*}
Similarly, $\dist(D_{18},(C_3)^2\sdp C_2)$ is achieved by an extension
of the example that yields $\dist(C_9,(C_3)^2)$.
The above is a complete catalogue of cases where
two groups are at distance precisely $2n^2/9$, except for the ad hoc
constructions for $\dist(C_9,C_9)$, $\dist(C_{18},C_{18})$
and $\dist(D_{18},D_{18})$ below.

Construction 2 can also be used directly to realize $\dist(C_{10},C_{10})$ and
$\dist(C_{14},C_{14})$.

\emph{Construction 3. (Ad hoc)}

$\dist(C_7,C_7)$: The distance between $C_7=\{0,\dots,6\}$ and its $(12)(56)$
isomorph is $18$.

$\dist(C_9,C_9)$: The distance between $C_9=\{0,\dots,8\}$ and its $(36)(47)(58)$
isomorph is $2n^2/9=18$.

Appropriate extensions of this last example realize
both $\dist(C_{18},C_{18})$ and $\dist(D_{18},D_{18})$.

\begin{remark} The computer calculations used in this paper were as follows: The graphs $\mathcal G'(8)$ and $\mathcal G'(16)$ were calculated by the first author using the \texttt{GAP} \cite{GAP} package \texttt{LOOPS} \cite{LOOPS} and modified code from \cite{VoEJC}. The inequalities of Section \ref{Sc:Inequalities} were independently verified by both authors, resulting in the list \eqref{Eq:Quadruples}. The algorithm for $m=2$ of Subsection \ref{Ss:m2} was implemented by both authors independently, and so was the algorithm for distances of cyclic groups of Subsection \ref{Ss:Cyclic}. The general algorithm for $\dist([\ga],[\gb])$ was run by the second author for all $n\le 22$ (which took several months on a single processor computer), and by the first author for $n\le 15$. Both authors verified the values $\mu_3(6)$--$\mu_3(10)$ of Proposition \ref{Pr:Graph} with independent programs. Finally, the upper bounds on $r$, $s$, $t$ and lower bounds on $u$ of Section \ref{Sc:ApplyGraph} were also performed independently by the two authors.
\end{remark}


\end{document}